\RequirePackage{fix-cm}
\documentclass[smallextended]{svjour3}       
\smartqed  
\usepackage{graphicx}
\usepackage{amssymb}
\usepackage{graphicx}
\usepackage{amsfonts}
\usepackage{mathrsfs}
\usepackage{amsmath}
\usepackage{amssymb}
\usepackage{tikz} 
\usepackage{csquotes}
\usepackage{subcaption}
\usepackage{algorithm}
\usepackage[noend]{algpseudocode}
\usepackage{bbm}
\usepackage{booktabs}
\usepackage[title]{appendix}

\newcommand{\R}{\mathbb{R}}
\newcommand{\N}{\mathbb{N}}
\newcommand{\Z}{\mathbb{Z}}

\newcommand{\conv}{\operatorname{conv}}

\newcommand{\supp}{\operatorname{supp}}

\newcommand{\st}{\text{s.t.}}

\newcommand{\knap}{K_{\mbox{\scriptsize \sc knap}}}
\newcommand{\C}{\mathscr{C}}
\newcommand{\dominates}{\triangleright}
\newcommand{\cA}{\mathcal{A}}
\renewcommand{\S}{\mathscr{S}}

\newtheorem{observation}{Observation}

\newenvironment{cpf}
{\begin{trivlist} \item[] {\em Proof of claim }}
{$\hfill\diamond$ \end{trivlist}}

\begin{document}

\title{Multi-cover Inequalities for Totally-Ordered Multiple Knapsack Sets: Theory and Computation
\thanks{A.~Del Pia is partially funded by ONR grant N00014-19-1-2322. Any opinions, findings, and conclusions or recommendations expressed in this material are those of the authors and do not necessarily reflect the views of the Office of Naval Research.}
}


\author{Alberto Del Pia \and Jeff Linderoth \and Haoran Zhu
}

\authorrunning{A. Del Pia et al.} 

\institute{Haoran Zhu \at
              Department of Industrial and Systems Engineering, University of Wisconsin-Madison, \\
              Madison, WI, 53705, USA\\
              hzhu94@wisc.edu           
}

\institute{Alberto Del Pia \at
              Department of Industrial and Systems Engineering, University of Wisconsin-Madison, \\
              Wisconsin Institute for Discovery,\\
              Madison, WI, 53705, USA\\
              delpia@wisc.edu
           \and
           Jeff Linderoth \at 
           Department of Industrial and Systems Engineering, University of Wisconsin-Madison, \\
              Wisconsin Institute for Discovery,\\
              Madison, WI, 53705, USA\\
              linderoth@wisc.edu
              \and
           Haoran Zhu \at
              Department of Industrial and Systems Engineering, University of Wisconsin-Madison, \\
              Madison, WI, 53705, USA\\
              hzhu94@wisc.edu 
}

\date{}

\maketitle

\begin{abstract}
We propose a method to generate cutting-planes from multiple covers of knapsack constraints.  The covers may come from different knapsack inequalities if the weights in the inequalities form a totally-ordered set.  Thus, we introduce and study the structure of a totally-ordered multiple knapsack set.  The valid \emph{multi-cover inequalities} we derive for its convex hull have a number of interesting properties. First, they generalize the well-known $(1,k)$-configuration inequalities. Second, they are not aggregation cuts.  Third, they cannot be generated as a rank-1 Chv{\'a}tal-Gomory cut from the inequality system consisting of the knapsack constraints and all their minimal cover inequalities.  We also provide conditions under which the inequalities are facets for the convex hull of the totally-ordered knapsack set, as well as conditions for those inequalities to fully characterize its convex hull.  We give an integer program to solve the separation and provide numerical experiments that showcase the strength of these new inequalities.


\end{abstract}

\section{Introduction}
\label{sec:intro}

In this paper we study cutting-planes related to covers of 0-1 knapsack sets.  
A \emph{0-1 knapsack set} is a set of the form 
\[ \knap := \{x \in \{0,1\}^n \mid a^T x \leq b\}, \] 
with $(a, b) \in \Z^{n+1}_+$, and a \emph{cover} is a subset $C \subseteq \{1, \ldots, n\}$ such that $\sum_{j \in C} a_j > b$.  
The associated \emph{cover inequality (CI)}
\[ \sum_{j \in C} x_j \leq |C| - 1 \]
is valid for the \emph{knapsack polytope} $\conv(\knap)$ and is not satisfied by the incidence vector of $C$.
There is a long and rich literature on (lifted) cover inequalities for the knapsack polytope
\cite{balas1975facets,hammer1975facet,wolsey1976facets,gu1998lifted,letchford2019lifted}, and the reader is directed to the recent survey \cite{hojny2019knapsack} for a thorough introduction.

In this paper we consider the \emph{0-1 multiple knapsack set}
\begin{align}
K = \{x \in \{0,1\}^n \mid Ax \leq b\}, \label{eq:tomks}
\end{align}
where $[A, b] \in \Z^{m \times (n+1)}_+$.
A standard and computationally useful way for generating  valid inequalities 
for $\conv(K)$
is to generate \emph{lifted cover inequalities (LCIs)} for the knapsack sets defined by the individual constraints of $K$
\cite{crowder1983solving}.  In this way, the extensive literature regarding valid inequalities for $\conv(\knap)$ can be leveraged  to solve integer programs whose feasible region is $K$.
In contrast to $\conv(\knap)$, very little is known about 
the polyhedral structure of $\conv(K)$.
In this paper, we introduce a family of cutting-planes, called \emph{multi-cover inequalities (MCIs)}, that are derived by simultaneously considering multiple covers that satisfy a certain condition.
The covers may come from \emph{any} inequality in the formulation, as long as the weights appearing in the knapsack inequalities are totally-ordered. Moreover, when only a single cover is given, the associated MCI reduces to a CI.  

More formally, we present a new approach to generate valid inequalities for a special multiple knapsack set, called the \emph{totally-ordered multiple knapsack set (TOMKS)}.
The multiple knapsack set $K$ in \eqref{eq:tomks} is called a TOMKS if the column vectors $\{A_{\cdot 1}, A_{\cdot 2}, \ldots , A_{\cdot n}\}$ of the constraint matrix $A \in \Z^{m \times n}_+$ form a chain ordered by component-wise order. Without loss of generality we may assume $A_{\cdot 1} \geq  A_{\cdot 2} \geq \ldots \geq A_{\cdot n}$.

The TOMKS can arise in the context of chance-constrained programming.  Specifically, consider a knapsack constraint where the weights of the items $(a)$ depend on a random variable $(\xi$), and we wish to satisfy the chance constraint
\begin{equation}
  \mathbb{P} \{ a(\xi)^T x \leq \beta \} \geq 1-\epsilon, \label{eq:chance}
\end{equation}
selecting a subset of items ($x \in \{0,1\}^n$) so that the likelihood that these items fit into the knapsack is sufficiently high.
In the scenario approximation approach proposed in \cite{MR2232597,nemirovski.shapiro:05}, an independent Monte Carlo sample of $N$ realizations of the weights $(a(\xi^1), \ldots a(\xi^N))$ is drawn and the deterministic constraints
\begin{equation}
  a(\xi^i)^T x \leq \beta \quad \forall i=1\ldots,N \label{eq:saa} 
\end{equation} 
are enforced.   
In \cite{MR2425035} it is shown that for every $\delta \in (0,1)$, if the sample size $N$ is sufficiently large:
\[ N \geq \frac{1}{2 \epsilon^2} \left( \log \left( \frac{1}{\delta} \right) + n \log(2) \right), \]
then any feasible solution to \eqref{eq:saa} satisfies the constraint \eqref{eq:chance} with probability at least $1-\delta$.
If the random weights of the items $a_1(\xi), a_2(\xi), \ldots a_n(\xi)$ are independently distributed with means $\mu_1 \geq \mu_2 \ldots \geq \mu_n$, then the feasible region in~\eqref{eq:saa} may either be a TOMKS, or the constraints can be (slightly) relaxed to obey the ordering property.

The TOMKS may arise in more general situations as well.  For a general binary set, if two knapsack inequalities $a_1^T x \leq b_1$ and $a_2^T x \leq b_2$ have non-zero coefficients in few common variables, their intersection may be totally-ordered, and the MCI would be applicable in this case.  In the general case, if variables are fixed to zero or one, the induced face may be a TOMKS, and the resulting inequalities could then be lifted to be valid for the original set. 
In the special case where the multiple covers come from the same knapsack set, the MCI can also produce interesting inequalities. 
For example, the well-known $(1,k)$-\emph{configuration inequalities} for $\conv(\knap)$ \cite{padberg19801} are a special case of MCI where two covers come from the same inequality and have particular structure (see Proposition~\ref{prop: 1-k-AMCI}). 
We also give an example where a facet of $\conv(\knap)$ generated by a recent lifting procedure described by Letchford and Souli \cite{letchford2019lifted} is a MCI.

Interestingly, as observed by Laurent and Sassano \cite{laurent1992characterization}, the convex hull of a TOMKS can be exactly characterized by all the associated minimal CIs if and only if the set of minimal covers has no minor isomorphic to $J_q = \{\{2,\ldots, q\}, \{1,i\} \text{ for }  i = 2, \ldots, q\}$ with $q \geq 3$, where the definition of \emph{minor} can be found in Sect.~\ref{sec:facet}. 
When the set of minimal covers does have a minor isomorphic to $J_q$, our new inequalities are important. In particular, when the minimal cover set is exactly $J_q, \conv(K)$ can be fully described by bound constraints, CIs, and a single MCI. 

MCIs are generated by a simple algorithm (Algorithm~\ref{alg:mci}) that takes as input a special family of covers $\C = \{C_1, C_2, \ldots C_k\}$ that obeys a certain maximality criterion (defined in Definition~\ref{def:multi-cover}).  For many types of cover-families $\C$, the MCI may be given in closed-form.
We also give a condition under which an MCI defines a facet of $\conv(K)$ in Section~\ref{sec:facet}, as well as a condition for the MCIs to fully describe $\conv(K)$. 

MCIs may be generated by simultaneously considering multiple knapsack inequalities defining $K$.
Another mechanism to generate inequalities taking into account information from multiple constraints of the formulation is to aggregate inequalities together, forming the set
\[ \cA(A,b) := \bigcap_{\lambda \in \R^m_+} \conv(\{x \in \{0,1\}^n \ \mid \lambda^T A x \leq \lambda^T b \}. \]
Inequalities valid for $\cA(A,b)$ are known as \emph{aggregation cuts}, and have been shown to be quite powerful from both an empirical \cite{fukasawa.goycoolea:11} and theoretical \cite{MR3844541} viewpoint.  
The well-known Chv{\'a}tal-Gomory (CG) cuts, lifted cover inequalities, and weight inequalities \cite{weismantel19970} are all aggregation cuts.  In Example~\ref{ex3}, we show that MCIs are not aggregation cuts.  
Furthermore, in Example~\ref{ex4}, we show that MCIs cannot be obtained as a (rank-1)  Chv{\'a}tal-Gomory cut from the linear system consisting of all minimal cover inequalities from $K$.

The paper is structured as follows: In Section~\ref{sec:dominance}, we define a certain type of dominance relationship between covers that is necessary to introduce the MCIs.  
The MCIs are defined in Section~\ref{sec:MCI}, where we also provide some examples to demonstrate that MCIs are not dominated by other well-known families of cutting-planes.  
Utilizing a combinatorial property of the multi-cover, in Section~\ref{subsec: extended_MCI} we introduce a strengthening procedure for MCI.
In Section~\ref{sec:facet} we provide a sufficient condition for the MCI to be facet-defining for $\conv(K)$, and we show a family of instances where the MCI inequalities are the only non-trivial inequalities besides cover inequalities for $\conv(K)$. 
Lastly in Section~\ref{sec: separation} and Section~\ref{sec: numerical} we discuss the separation problem for MCIs, and we present numerical experiments showing the additional integrality gap that can be closed by MCIs when compared to CIs.

Next, we detail the differences between this paper and the preliminary IPCO version \cite{del2021multi}. 
First, the MCI defined in \cite{del2021multi} is referred to as the \emph{simple-MCI} in this paper, while the MCI defined in this paper refers to a larger, more general, class of inequalities.  This family contains both the \emph{simple-MCI} and the antichain multi-cover inequalities \emph{AMCI} in \cite{del2021multi} as special cases. 
For that reason, we remove the \emph{AMCI} section in \cite{del2021multi} from this paper.
Moreover, in Section~\ref{sec:MCI} we introduce an analogous concept of \emph{extended cover inequality (ECI)} for MCI, namely the \emph{extended MCI}, and in Section~\ref{sec:facet} we provide a special scenario for MCIs to fully describe $\conv(K)$.  Lastly, Section~\ref{sec: separation} and Section~\ref{sec: numerical} have completely new results about separation and numerical experiments.

\paragraph{Notation.}
For a positive integer $n$, we denote by $[n] := \{1,\dots,n\}$.
The \emph{incidence vector} of a set $S \subseteq [n]$ is denoted by $\chi^S$.
Therefore, given a TOMKS $K$, we say that a set $S \subseteq [n]$ is a \emph{cover} for $K$ if $\chi^S \notin K$. 
For a vector $x \in \R^n$ and a set $S \subseteq [n]$, we define $x(S) := \sum_{i \in S} x_i$. 
This in particular means $x(\emptyset) = 0.$ 
We denote the \emph{power set} of a set $S$, the set of all subsets of $S$, as $2^S$.
Lastly, we denote $e^j$ to be the $j$-th unit vector of the ambient space.

\section{A Dominance Relation}
\label{sec:dominance}

In this section we define a dominance relation between covers and show some of its important properties.

\begin{definition}[Dominance]
\label{def:domination}
For $S_1, S_2 \subseteq [n]$, we say that $S_1$ \emph{dominates} $S_2$ and write $S_1 \dominates S_2$, if there exists an injective function $f: S_2 \to S_1$ with $f(i) \leq i \ \forall i \in S_2$. 
\end{definition}

The dominance relation in Definition~\ref{def:domination} is reflexive, antisymmetric, and transitive, so $(2^{[n]}, \dominates)$ forms a partially ordered set. 
For two sets $S_1, S_2 \subseteq [n]$, we say $S_1$ and $S_2$ are \emph{comparable} if $S_1 \dominates S_2$ or $S_2 \dominates S_1$.

The dominance relation has a natural use in the context of covers.
In fact, if $C_2$ is a cover for a TOMKS $K$ and $C_1$ dominates $C_2$, then $C_1$ is also a cover for $K$. 
Next, we present two technical lemmas about dominance.

\begin{lemma}
\label{lem: domination_minus_intersection}
Let $S_1, S_2 \subseteq [n]$ with $S_1 \neq S_2$.
Then for any $S' \subseteq S_1 \cap S_2$, $S_1 \dominates S_2$ if and only if $S_1 \setminus S' \dominates S_2 \setminus S'$.
\end{lemma}

\begin{proof}
Let $S_0 := S_1 \cap S_2$. First, we show that for any $S' \subseteq S_0$, $S_1 \setminus S' \dominates S_2 \setminus S'$ implies that $S_1 \dominates S_2$. 
From Definition~\ref{def:domination}, there exists an injective function $g: S_2 \setminus S' \to S_1 \setminus S'$ such that $g(i) \leq i \ \forall i \in S_2 \setminus S'$. Now we define the function $f$ as follows: for any $i \in S_2$, if $i \in S'$, then $f(i): = i$; if $i \in S_2 \setminus S'$, then $f(i): = g(i)$. So $f$ is a function maps elements from $S_2$ to $S_1$. Apparently this function $f$ is an injective function, with $f(i) \leq i \ \forall i \in S_2$, therefore we have proven that $S_1 \dominates S_2$. 

Next we show that $S_1 \dominates S_2$ implies $S_1 \setminus S_0 \dominates S_2 \setminus S_0$. From Definition~\ref{def:domination}, there exists an injective function $g : S_2 \to S_1$ such that $g(i) \leq i \ \forall i \in S_2$. 
We define the function $f$ as follows: for any $i \in S_2 \setminus S_0$, $f(i) := g^{t(i)}(i)$ where $t(i)$ is the smallest integer number $t \ge 1$ such that $g^t(i) \in S_1 \setminus S_0$, and $g^t = g \circ g \circ \ldots \circ g$ is the $t$th functional power of $g$. 
Since $g$ is an injective function from $S_2$ to $S_1$, we know for any $i \in S_2 \setminus S_0$, there must exist $t \in \N$ such that $g^t(i) \in S_1 \setminus S_0$. 
Hence the function $f$ is well defined and it is from $S_2 \setminus S_0$ to $S_1 \setminus S_0$. It remains to show that $f$ is an injective function and that $f(i) \leq i \ \forall i \in S_2 \setminus S_0$. 
The fact that $f(i) \leq i \ \forall i \in S_2 \setminus S_0$ follows directly from the property $g(i) \leq i \ \forall i \in S_2$ and the definition of $f$. 
Finally, we show that $f$ is injective.
Assume, for a contradiction, that $f$ is not injective.
Then there exists $i, j \in S_2 \setminus S_0$ with $i \neq j$ such that $f(i) = f(j)$, i.e., $g^{t(i)}(i) = g^{t(j)}(j)$. 
By eventually switching $i$ and $j$, we can assume without loss of generality that $t(i) \leq t(j)$. 
Since $g$ is injective, we know that $i = g^{t(j) - t(i)} (j)$. 
If $t(j) - t(i) = 0$, then $i = j$ which gives us a contradiction. 
Thus we now assume $t(j) - t(i) \geq 1.$ 
In this case, the fact that $g^{t(j) - t(i)} (j)$ is in $S_2 \setminus S_0$ contradicts the fact that the codomain of $g$ is $S_1$.
We have thereby shown that $f$ is injective, and this concludes the proof of $S_1 \setminus S_0 \dominates S_2 \setminus S_0$.

Lastly, we want to show that, for any $S' \subseteq S_0$, $S_1 \dominates S_2$ implies that $S_1 \setminus S' \dominates S_2 \dominates S'$. If $S' = S_0$, then from above, we have already shown that $S_1 \setminus S_0 \dominates S_2 \setminus S_0$. If $S' \subsetneq S_0,$ let $S'' = S_0 \setminus S'$, then $S_1 \setminus S' = (S_1 \setminus S_0) \cup S'', S_2 \setminus S' = (S_2 \setminus S_0) \cup S''$. 
From $S_1 \setminus S_0 \dominates S_2 \setminus S_0$, we can easily know that $(S_1 \setminus S_0) \cup S'' \dominates (S_2 \setminus S_0) \cup S''$, which is just $S_1 \setminus S' \dominates S_2 \setminus S'$. 
\qed \end{proof}

\begin{lemma}
\label{lem: domination_dual}
Let $S \subseteq [n]$ and let $S_1 ,S_2 \subseteq S$.
Then, $S_1\dominates S_2$ if and only if $S \setminus S_2 \dominates S \setminus S_1$. 
\end{lemma}

\begin{proof}
It suffices to prove the implication from left to right, i.e., that $S_1\dominates S_2$ implies $S \setminus S_2 \dominates S \setminus S_1$.
In fact, once this implication is proven, it is simple to observe that the reverse implication holds as well. 
To see this, assume $S \setminus S_2 \dominates S \setminus S_1$.
Using the implication from left to right, we then obtain $S \setminus (S \setminus S_1) \dominates S \setminus (S \setminus S_2)$, which can be rewritten as $S_1 \dominates S_2$.
Hence, in the remainder of the proof we show the implication from left to right.



From Definition~\ref{def:domination}, we know that there exists an injective function $f: S_2 \to S_1$ with $f(i) \leq i \ \forall i \in S_2$.
Define $\tilde S_1 := f(S_2)$.
To complete the proof, it suffices to show that $S \setminus S_2 \dominates S \setminus \tilde S_1$.
In fact, since $\tilde S_1 \subseteq S_1$, then $S \setminus \tilde S_1 \supseteq S \setminus S_1$, thus using the identity function we obtain $S \setminus \tilde S_1 \dominates S \setminus S_1$.
Since the domination relation is transitive, we then obtain $S \setminus S_2 \dominates S \setminus \tilde S_1 \dominates S \setminus S_1$, as desired.
Hence we now show that that $S \setminus S_2 \dominates S \setminus \tilde S_1$.

Note that $|\tilde S_1| = |S_2|$, thus also $|S \setminus \tilde S_1| = |S \setminus S_2|$, and we denote the latter cardinality by $t$.
Let $S \setminus \tilde S_1 = \{i_1, \ldots, i_t\}$ and $S \setminus S_2 = \{j_1, \ldots, j_t\}$, where $i_1 < \ldots < i_t$ and $j_1 < \ldots < j_t$. 
It suffices to show that $j_h \leq i_h \ \forall h \in [t]$.
In fact, then we can consider the injective function $g: S \setminus \tilde S_1 \to S \setminus S_2$ defined by $g(i_h) := j_h$, and obtain that $S \setminus S_2 \dominates S \setminus \tilde S_1$.
Thus, in the remainder of the proof we show that $j_h \leq i_h \ \forall h \in [t]$.
We prove this statement by contradiction, thus we suppose that there exists at least one index $h$ such that $j_h > i_h$, and we define $h^*$ to be the minimum such index, i.e., $h^*: = \min\{h \mid j_h > i_h\}$.

We now show that $\{j \in S \setminus S_2 \mid j \leq i_{h^*}\} = \{j_1, \ldots, j_{h^* - 1}\}$.
The containment $\subseteq$ holds because, by definition of $h^*$, we have $j_{h^*} > i_{h^*}$.
To prove the containment $\supseteq$, we just need to show that $j_{h^* - 1} \le i_{h^*}$.
If not, we have $j_{h^* - 1} > i_{h^*}$, thus $j_{h^* - 1} > i_{h^*} > i_{h^* - 1}$, which contradicts the choice of $h^*$.



Now we consider the sets $T_1: = \{j \in \tilde S_1 \mid j \leq i_{h^*}\}$ and $T_2: = \{j \in S_2 \mid j \leq i_{h^*}\}$. 
Since $\tilde S_1,S_2 \subseteq S$, we have $\tilde S_1 = S \setminus (S \setminus \tilde S_1)$ and $S_2 = S \setminus (S \setminus S_2)$.
We obtain 
\begin{align*}
|T_1| & = |\{j \in S \mid j \leq i_{h^*}\}| - |\{j \in S \setminus \tilde S_1 \mid j \leq i_{h^*}\}|  
= |\{j \in S \mid j \leq i_{h^*}\}| - h^*, \\
|T_2| & = |\{j \in S \mid j \leq i_{h^*}\}| - |\{j \in S \setminus S_2 \mid j \leq i_{h^*}\}| 
= |\{j \in S \mid j \leq i_{h^*}\}| - h^* + 1,
\end{align*}
therefore $|T_1| < |T_2|$.
To obtain a contradiction we now show $|T_1| \ge |T_2|$.
Since $f(S_2) = \tilde S_1$ and $f(j) \leq j \ \forall j \in S_2$, we obtain 
\begin{align*}
|T_1| = |\{j \in \tilde S_1 \mid j \leq i_{h^*} \}| = |\{j \in f(S_2) \mid j \leq i_{h^*}\}| \geq |\{j \in S_2 \mid j \leq i_{h^*}\}| = |T_2|.
\end{align*}
We have derived a contradiction.
Therefore, we have shown that $j_h \leq i_h \ \forall h \in [t]$, and this concludes the proof.
\qed \end{proof}

\section{Multi-cover Inequalities and Variation}
\label{sec:MCI}

Throughout this section, we consider a TOMKS $K : = \{x \in \{0,1\}^n \mid Ax \leq b\}$, and we introduce the \emph{multi-cover inequalities (MCIs)}, which form a novel family of valid inequalities for $K$.
Each MCI can be obtained from a special family of covers $\{C_1, \ldots, C_k\}$ for $K$ that we call a multi-cover.
In order to define a multi-cover,
we first introduce the \emph{discrepancy family}.
\begin{definition}[Discrepancy family]
For a family of sets $\C = \{C_1, \ldots, C_k\}$, we say that $\{C_1 \setminus \cap_{h=1}^k C_h, \ldots, C_k \setminus \cap_{h=1}^k C_h\}$ is the \emph{discrepancy family} of $\C$, and we denote it by $\mathcal D(\C)$.
\end{definition}

%
%
%
%
%
Now we can define the concept of a multi-cover.



\begin{definition}[Multi-cover]
\label{def:multi-cover}
Let $\C$ be a family of covers for $K$.
We then say that $\C$ is a \emph{multi-cover} for $K$ if for any set $T \subseteq \cup_{D \in \mathcal D(\C)} D$ with $T \notin \mathcal{D}(\C)$, there exists some $D' \in \mathcal{D}(\C)$ such that $T \dominates D'$ or $D' \dominates T$.
\end{definition}

\begin{example}
\label{exm: 1}
Consider the TOMKS:
{\small
\begin{align*}
K: = \{x \in \{0,1\}^5 \mid \ & 19 x_1 + 11x_2 + 5x_3 + 4x_4 + 2x_5 \leq 31, \\
& 16 x_1 + 10 x_2 + 7 x_3 + 5 x_4 + 3 x_5  \leq 30\}.
\end{align*}
}
We have that $\C = \{C_1, C_2\} := \{\{1,2,5\},\{1,3,4,5\}\}$ is a multi-cover.
In fact, it is simple to check that $C_1$ and $C_2$ are covers for $K$.
Furthermore, the discrepancy family of $\C$ is $\mathcal{D}(\C) = \{\{2\}, \{3,4\}\}$, and for any $T \subseteq \{2,3,4\}$, if $|T| = 1$ we have $\{2\} \dominates T$, while if $|T| \geq 2$ we have $T \dominates \{3,4\}$.
$\hfill\diamond$
\end{example}


For a given family of covers $\{C_1, \ldots, C_k\}$ for $K$, 
throughout this paper,
for ease of notation we define 
$C_0 := \cap_{h=1}^k C_h$, 
$C := \cup_{h=1}^k C_h$, 
$\bar C_h: = C \setminus C_h$ for $h \in [k]$,
and similarly $\bar T := C \setminus T$ for any $T \subseteq C$. 

We are now ready to introduce our multi-cover inequalities for $K$.
These inequalities are defined by the following algorithm.



\begin{algorithm}[H]
   \caption{Multi-cover inequality (MCI)}
   \label{alg:mci}
       \hspace*{\algorithmicindent} \textbf{Input:} 
       A multi-cover $\{C_1, \ldots, C_k\}$ for $K$. \\
       \hspace*{\algorithmicindent} \textbf{Output:} 
       A multi-cover inequality.
    \begin{algorithmic}[1]
    \State Let 
    $\{i_1, \ldots, i_m\} : = C \setminus C_0$, with $i_1 < \ldots < i_m$.
    \State Set $\alpha_i := 1$ if $i \in  \{i_1, \ldots, i_m\}$, and $\alpha_i := 0$ otherwise. \label{step: MCI_2}
    \For {$t=m-1, \ldots, 1$} \label{step: MCI_3}
        \State 
        let $\alpha_{i_t}$ be \textit{any} integer $\geq \max_{h \in [k]: i_t \in C_h} \max_{\ell \in \bar C_h: \ell>i_t} \alpha_\ell + 1$.
         \label{step: MCI_4}
    \EndFor
\For {$j \in C_0$}
\State let $\alpha_j$ be \textit{any} integer $\geq \min_{h \in [k]} \max \big\{ \max_{\ell<j, \ell \in \bar C_h} \alpha_\ell, \sum_{\ell>j, \ell \in \bar C_h} \alpha_\ell + 1 \big\}$.	\label{step: modified_coe_intersection}
\EndFor
    \State Set $\beta := \max_{h=1}^k \alpha(C_h) - 1.$ \label{step: MCI_7}\\
      \Return the inequality $\alpha^T x \leq \beta$.
\end{algorithmic}
\end{algorithm}


We remark that in Algorithm~\ref{alg:mci},
in the case where we take the minimum or maximum over an empty set (see Step~\ref{step: MCI_4} and \ref{step: modified_coe_intersection}), the corresponding minimum or maximum should be set to zero.

For the above algorithm, we have the following easy observation.

\begin{observation}
Given a multi-cover $\{C_h\}_{h=1}^k$, Algorithm~\ref{alg:mci} performs a number of operations that is polynomial in $|C|$ and $k$.
Furthermore, $\supp(\alpha) = C$.
\end{observation}

At Step~\ref{step: MCI_4} and \ref{step: modified_coe_intersection} of Algorithm~\ref{alg:mci}, when the $\geq$ is fixed to $=$, the obtained MCI is called the \emph{simple-MCI}, which was given in \cite{del2021multi}. 
Note that given a multi-cover, we can obtain several MCIs, but only one simple-MCI. 

\begin{definition}[simple-MCI]
Given a multi-cover $\{C_1, \ldots, C_k\}$ and one of its MCIs $\alpha^T x \leq \beta$, if for every $i \in C \setminus C_0, \alpha_i = \max_{h \in [k]: i \in C_h} \max_{\ell \in \bar C_h: \ell>i} \alpha_\ell + 1$ and for every $j \in C_0, \alpha_j = \min_{h \in [k]} \max \big\{ \max_{\ell<j, \ell \in \bar C_h} \alpha_\ell, \sum_{t>j, t \in \bar C_h} \alpha_t + 1 \big\}$, then such MCI $\alpha^T x \leq \beta$ is called a \emph{simple-MCI (S-MCI)}.
\end{definition}

The main result of this section is that, given a multi-cover for $K$, then any of its MCI is valid for $\conv(K)$. Before presenting the theorem, we will need the following auxiliary result.

\begin{proposition}
\label{prop: key_prop_valid}
Let $\{C_h\}_{h=1}^k$ be a multi-cover and let $\alpha^T x \leq \beta$ be one of its associated MCIs. 
If there exists $T \subseteq C \setminus C_0$, with $T \notin \{\bar C_h\}_{h=1}^k$ and $T \dominates \bar C_{h'}$ for some $h' \in [k]$, then $\alpha(T) > \alpha(\bar C_{h'})$.
\end{proposition}

 We remark that Proposition~\ref{prop: key_prop_valid} does not depend on the specific property of multi-covers.  That is, it holds for any family of covers.

 \begin{proof}
Let $T$ and $\bar C_{h'}$ be the sets as assumed in the statement of this proposition, with $T \dominates \bar C_{h'}$.
Let $T_0: = T \cap \bar C_{h'}$, $T_1: = T \setminus T_0,$ and $T_2: = \bar C_{h'} \setminus T_0$. 
Then $T = T_0 \cup T_1$, $\bar C_{h'} = T_0 \cup T_2$. 
Since $T \neq \bar C_{h'}$ and $T \dominates \bar C_{h'}$, then we know $T_1 \neq \emptyset$. By Lemma~\ref{lem: domination_minus_intersection}, we know that $T_1 \dominates T_2$. 
If $T_2 = \emptyset$, then $\alpha(T) = \alpha(T_0) + \alpha(T_1) > \alpha(T_0) = \alpha(\bar C_{h'})$. Hence we assume $T_2 \neq \emptyset$. 
Let $T_2: = \{j_1, \ldots, j_t\}$. Since $T_1 \dominates T_2$ and $T_1 \cap T_2 = \emptyset$, we know there exists $\{k_1, \ldots, k_t\} \subseteq T_1$ such that $k_1 < j_1, \ldots, k_t < j_t$. 

W.l.o.g., consider $k_1$ and $j_1$. 
By definition, there is $k_1 < j_1, k_1 \notin \bar C_{h'}, j_1 \in \bar C_{h'}$, or equivalently: $k_1 < j_1, k_1 \in C_{h'}, j_1 \in \bar C_{h'}$. Therefore, $j_1 \in \{\ell \mid \ell > k_1, \ell \in \bar C_{h'}, k_1 \in C_{h'}\}$. 
By construction of MCI, we know that $\alpha_{k_1} > \alpha_{j_1}$. 
For the remaining $j_2$ and $j_2, \ldots, k_t$ and $j_t$, the same argument yields $\alpha_{k_2} > \alpha_{j_2}, \ldots, \alpha_{k_t} > \alpha_{j_t}$.


Therefore, $\alpha(T) = \alpha(T_1) + \alpha(T_0) \geq \alpha_{k_1} + \ldots + \alpha_{k_t} + \alpha(T_0) > \alpha_{j_1} + \ldots + \alpha_{j_t} + \alpha(T_0) = \alpha(T_2) + \alpha(T_0) =  \alpha(\bar C_{h'})$, which concludes the proof. 
 \qed \end{proof}

Now we are ready to present the first main result of this paper.

\begin{theorem}
\label{theo: MCI_antichain}
Given a multi-cover $\{C_h\}_{h=1}^k$ for a TOMKS $K$, then the corresponding MCIs are valid for $\conv(K)$. 
\end{theorem}

\begin{proof}
Since $\supp(\alpha) = C$, in order to show that $\alpha^T x \leq \beta$ is valid for $\conv(K)$, it suffices to show that, for any $T \subseteq C$ with $\alpha(T) \geq \beta + 1$, $T$ must be a cover for $K$. Note that by Step~\ref{step: MCI_7} of Algorithm~\ref{alg:mci}, we have $\beta + 1= \max_{h=1}^k \alpha(C_h)$, 
and for any $T_1, T_2 \subseteq C$, $\alpha(T_1) \geq \alpha(T_2)$ is equivalent to $\alpha(\bar T_1) \leq \alpha(\bar T_2)$.
Therefore from Lemma~\ref{lem: domination_dual}, it suffices to show that: for any $T \subseteq C$ with $\alpha(\bar T) \leq \min_{h=1}^k\alpha(\bar C_h)$, there must exist some $h^* \in [k]$ such that $\bar C_{h^*} \dominates \bar T$. We will assume that $T \notin \{C_h\}_{h=1}^k$ since otherwise $\bar C_{h^*} \dominates \bar T$ trivially holds. 
In the following, the proof is subdivided into two cases, depending on whether $\bar T \cap C_0 = \emptyset$ or not.

\smallskip

First, we consider the case $\bar T \cap C_0 = \emptyset$.
In this case, we have $C_0 \subseteq T$. 
%
%
by Definition~\ref{def:multi-cover} of multi-cover, we know there must exist $h^* \in [k]$ such that either $C_{h^*} \setminus C_0 \dominates  T \setminus C_0$, or $ T \setminus C_0 \dominates C_{h^*} \setminus C_0$. By the above assumption $C_0 \subseteq T$ and Lemma~\ref{lem: domination_minus_intersection}, we know that either $C_{h^*} \dominates T$ or $T \dominates C_{h^*}$. If $T \dominates C_{h^*}$, then Lemma~\ref{lem: domination_dual} implies $\bar C_{h^*} \dominates \bar T$, which completes the proof. 
So we assume $C_{h^*} \dominates T$, or equivalently, $\bar T \dominates \bar C_{h^*}$. Since $\bar T \subseteq C \setminus C_0$ and $\bar T \neq \bar C_{h^*}$, by Proposition~\ref{prop: key_prop_valid} we obtain that $\alpha(\bar T) > \alpha(\bar C_{h^*})$, and this contradicts to the assumption that $\alpha(\bar T) \leq \min_{h=1}^k\alpha(\bar C_h)$. 

\smallskip

Next, we consider the case $\bar T \cap C_0 \neq \emptyset$.
In this case, we want to construct a $\bar D \subseteq C$ with $\bar D \cap C_0 = \emptyset, \alpha(\bar D) \leq \alpha(\bar T)$, and $\bar D \dominates \bar T$. Then since $\alpha(\bar T) \leq \min_{h=1}^k\alpha(\bar C_h)$, we have $\alpha(\bar D) \leq \min_{h=1}^k\alpha(\bar C_h)$ where $\bar D \cap C_0 = \emptyset$. According to our discussion in the previous case, we know that there exists some $h^* \in [k]$ such that $\bar C_{h^*} \dominates \bar D$, which implies $\bar C_{h^*} \dominates \bar T$ since $\dominates$ forms a partial order, and the proof is completed.

Next, we arbitrarily pick $t^* \in \bar T \cap C_0$. 
Then by Step~\ref{step: modified_coe_intersection} of Algorithm~\ref{alg:mci}, we know that there exists $h^* \in [k]$, such that:
$$\alpha_{t^*} \geq \max \big\{ \max_{\ell <t^*, \ell \in \bar C_{h^*}} \alpha_\ell, \sum_{t>t^*, t \in \bar C_{h^*}} \alpha_t + 1 \big\}.
$$
If $\{\ell \in \bar C_{h^*} \mid \ell < t^*\} \subseteq \bar T$,
then we have $\alpha(\bar T) \geq \sum_{\ell < t^*, \ell \in \bar C_{h^*}} \alpha_\ell + \alpha_{t^*}$, which is at least $\sum_{\ell < t^*, \ell \in \bar C_{h^*}} \alpha_\ell + \sum_{t>t^*, t \in \bar C_{h^*}} \alpha_t + 1$. Since $t^* \notin \bar C_{h^*}$, we know that $\sum_{\ell < t^*, \ell \in \bar C_{h^*}} \alpha_\ell + \sum_{t>t^*, t \in \bar C_{h^*}} \alpha_t + 1 = \alpha(\bar C_{h^*}) + 1$. Hence $\alpha(\bar T) > \alpha(\bar C_{h^*})$, and this contradicts the initial assumption $\alpha(\bar T) \leq \min_{h=1}^k\alpha(\bar C_h)$. 
Therefore we can find some $\ell^* \in \bar C_{h^*}, \ell^* < t^*$ such that $\ell^* \notin \bar T$. 
Now define $\bar D: = \bar T \cup \{\ell^*\} \setminus \{t^*\}$. 
Since $\ell^* < t^*$, clearly $\bar D \dominates \bar T$. 
Also $\alpha(\bar T) - \alpha(\bar D) = \alpha_{t^*} - \alpha_{\ell^*}$, since $\alpha_{t^*} \geq \max_{\ell<t^*, \ell \in \bar C_{h^*}} \alpha_\ell$, we know that $\alpha(\bar T) - \alpha(\bar D) \geq 0$. 
If $\bar D \cap C_0 = \emptyset$, then we are done. 
Otherwise, we can replace $\bar T$ by $\bar D$, consider any index in $\bar D \cap C_0$ and apply once more the above discussion.
Every time we are able to obtain a set $\bar D$ with $|\bar D \cap C_0|$ decreased by 1.
At the end we will obtain a set $\bar D$ with the desired property: $\bar D \cap C_0 = \emptyset, \alpha(\bar D) \leq \alpha(\bar T)$, and $\bar D \dominates \bar T$.
This completes the proof of the case $\bar T \cap C_0 \neq \emptyset$.

\smallskip

Therefore from the discussion of the above two cases, we have concluded the proof that the MCI $\alpha^T x \leq \beta$ is a valid inequality for $\conv(K)$. 
\qed \end{proof}

\subsection{Illustrative examples}

In this section, we provide some examples to showcase the novelty of our cut-generating procedure.
We start with a simple example to help understand Algorithm~\ref{alg:mci}.

\begin{example}
\label{exam: explain}
Consider a multi-cover $\{C_1, C_2\} = \{\{1,2,3,4\}, \{1,2,4,5,6,7\}\}$, here $C = [7]$, $C_0 = \{1,2,4\}$, and $\{i_1, i_2, i_3, i_4\} = \{3,5,6,7\}$. 
Next, we compute the coefficients of one particular MCI.

The coefficient $\alpha_7$ is initialized as 1. 
For $t=3$, the coefficient $\alpha_6 \geq \max_{h \in [2]: 6 \in C_h} \max_{\ell \in \bar{C}_h: \ell > 6} \alpha_\ell + 1 = 1$, and we pick $\alpha_6 = 1$. 
Similarly, for $t=2$, we pick $\alpha_5 = 1$. 
For $t=1$, we have $\alpha_3 \geq \max_{h \in [2]: 3 \in C_h} \max_{\ell \in \bar{C}_h: \ell > 2} \alpha_\ell + 1 = 2$, and we pick $\alpha_3 = 3$. 
For $j \in C_0 = \{1,2,4\}$, from Step~\ref{step: modified_coe_intersection}, we have:
\begin{align*}
& \alpha_1 \geq \min_{h \in [2]} \max \{\max_{\ell < 1, \ell \in \bar{C}_h} \alpha_\ell, \sum_{\ell > 1, \ell \in \bar{C}_h} \alpha_\ell + 1\} = 4,\\
& \alpha_2 \geq \min_{h \in [2]} \max \{\max_{\ell < 2, \ell \in \bar{C}_h} \alpha_\ell, \sum_{\ell > 2, \ell \in \bar{C}_h} \alpha_\ell + 1\} = 4,\\
& \alpha_4 \geq \min_{h \in [2]} \max \{\max_{\ell < 4, \ell \in \bar{C}_h} \alpha_\ell, \sum_{\ell > 4, \ell \in \bar{C}_h} \alpha_\ell + 1\} = 3.
\end{align*}
Here we pick $\alpha_j$ as small as possible, for $j \in C_0$. 
Hence $\alpha = (4,4,3,3,1,1,1)$, and from Step~\ref{step: MCI_7}, we have $\beta = \max_{h \in [2]}\alpha(C_h) - 1 = 13$.
Therefore, we obtain the MCI $4x_1 + 4x_2 + 3x_3 + 3x_4 + x_5 + x_6 + x_7 \leq 13$. 
Note that such MCI is not an S-MCI, because at Step~\ref{step: MCI_4} for $t=1$, we selected $\alpha_{i_1} = 3$ instead of the lower bound $2$. 
$\hfill\diamond$
\end{example}

In fact, for some multi-covers with a specific discrepancy family, we are able to write some  MCIs in closed form. 
The next example can be seen as a generalization of the multi-cover in Example~\ref{exam: explain}. 

 \begin{example}
 \label{exam: 3}
 Consider $\{C_1, C_2\}$ with discrepancy family $\{\{i_1\}, \{i_2, \ldots, i_t\}\}$, with $i_1 < \ldots < i_{t}$ and $t \geq 3$. Here the family of covers $\{C_1, C_2\}$ is a multi-cover, and the following inequality is one of its associated MCIs:
 {\small
 \begin{equation}
 \label{eq: CAI_1}
 \sum_{i<i_1, i \in C} t x_i +  \sum_{i_1 \leq i < i_2,  i \in C} (t-1) x_i + \sum_{\ell=3}^{t} \sum_{i_{\ell-1} < i < i_\ell,  i \in C} (t-\ell+2) x_i + \sum_{\ell=2}^t x_{i_\ell} + \sum_{i>i_t,  i \in C} x_i \leq \beta,
 \end{equation}
 }
where $\beta$ is the left-hand-side term evaluated at the point $\chi^{C_1} (\text{or } \chi^{C_2})$  minus 1. 
 $\hfill\diamond$
 \end{example}
 
Here is the S-MCI formula for another multi-cover with a specific structure. 
 
\begin{example}
\label{exam: 1_t+1}
Consider $\{C_1, C_2\}$ with discrepancy family $\{\{i_1, i_{t+1}\}, \{i_2, \ldots, i_t\}\}$ for some $t \geq 3$, with $i_1 < \ldots < i_{t+1}$. 
It is simple to verify that $\{C_1, C_2\}$ is a multi-cover, and the obtained S-MCI is:
{\small
\begin{equation}
\label{eq: CAI_2}
\begin{split}
& \sum_{i<i_1, i \in C} (2t - 1) x_i + \sum_{i_1 \leq i < i_2,  i \in C} (2t - 3) x_i + \sum_{\ell=3}^{t} \sum_{i_{\ell-1} < i < i_\ell,  i \in C} (2t-2\ell+3) x_i + \\
&  + \sum_{\ell=2}^t 2 x_{i_\ell} + \sum_{i_t < i< i_{t+1},  i \in C} 2 x_i + x_{i_{t+1}} + \sum_{i > i_{t+1}, i \in C} 2x_i \leq \beta,
\end{split}
\end{equation}
}
where $\beta$ is the left-hand-side term evaluated at the point $\chi^{C_1} (\text{or }\chi^{C_2})$ minus 1. 
$\hfill\diamond$
\end{example}

Next, we provide another simple example of S-MCI. 

\begin{example}
\label{exam: closed_2}
Consider $\{C_1, C_2, C_3\}$ with discrepancy family $\{\{i_1, i_3\},\{i_1, i_4, i_5\},$ $\{i_2, i_3, i_5\}\}$, with $i_1 < \ldots < i_{5}$. Here the family of covers $\{C_1, C_2, C_3\}$ is a multi-cover, and the obtained S-MCI is:
{\small
\begin{equation}
\label{eq: CAI_3}
\begin{split}
& \sum_{i < i_1, i \in C}  5x_i + \sum_{i_1 \leq i < i_2, i \in C} 3x_i + 2 x_{i_2} +  \sum_{i_2 < i < i_3, i \in C} 3x_i + 2x_{i_3} + \\
& +  \sum_{i_3 < i < i_4, i \in C} 2x_i + x_{i_4} +  \sum_{i_4 < i < i_5, i \in C} 2x_i + x_{i_5} + \sum_{i > i_5, i \in C} 2x_i \leq \beta,
\end{split}
\end{equation}
}
where $\beta$ is the left-hand-side term evaluated at the point $\chi^{C_1} (\text{or } \chi^{C_2}, \chi^{C_3})$  minus 1. 
$\hfill\diamond$
\end{example}

Next, we present some illustrative examples to showcase the utility of MCIs. 
The first example shows that, unlike LCIs or CG cuts, S-MCIs are not
aggregation cuts for the original linear system.

\begin{example}
\label{ex3}
Consider the TOMKS $K$ and the multi-cover $\{C_1, C_2\}$ defined in Example~\ref{exm: 1}.
Note that point $\chi^{C_1}$ only violates the first knapsack constraint, and point $\chi^{C_2}$ only violates the second knapsack constraint.  
The associated S-MCI is 
\begin{equation}
    3 x_ 1 + 2 x_2 + x_3 + x_4 + x_5 \leq 5, \label{eq:ex3}
\end{equation}
and~\eqref{eq:ex3} is violated by both $\chi^{C_1}$ and $\chi^{C_2}$. 

One can further check that this S-MCI~\eqref{eq:ex3} is facet-defining for $\conv(K)$. In fact, $\conv(K)$ can be exactly characterized by this S-MCI, along with the bound constraints $0 \leq x_i \leq 1 \ \forall i \in [5]$, and the following four CIs: $x_1 + x_2 + x_5 \leq 2, x_1 + x_2 + x_4 \leq 2, x_1 + x_2 + x_3 \leq 2, x_1 + x_3 + x_4 + x_5 \leq 3$. 

Now consider an aggregation of the knapsack inequalities for $K$ given by inequality $\lambda_1 (19,11,5,4,2)^T x + \lambda_2 (16,10,7,5,3)^T x \leq 31\lambda_1 + 30 \lambda_2$, where $\lambda_1, \lambda_2 \geq 0$. 
For any choice of $\lambda_1 \geq 0, \lambda_2 \geq 0$, it can be verified that $C_1$ and $C_2$ cannot both be covers for the knapsack set given by this single inequality, so any aggregation cut for $K$ can cut off at most one vector among $\chi^{C_1}$ and $\chi^{C_2}$. Therefore, the inequality~\eqref{eq:ex3} is not an aggregation cut.  
In some cases, it may be possible to obtain an S-MCI as a CG cut for the original linear system \emph{augmented} with its minimal cover inequalities.  
In this example, consider the set
{\small
\begin{align*}
K_{CI} : = \{x \in \{0,1\}^5 \mid \ & 19 x_1 + 11x_2 + 5x_3 + 4x_4 + 2x_5 \leq 31, \\ 
& 16 x_1 + 10 x_2 + 7 x_3 + 5 x_4 + 3 x_5  \leq 30,\\
& x_1 + x_2 + x_3 \leq 2, x_1 + x_2 + x_4 \leq 2,\\  
& x_1 + x_2 + x_5 \leq 2, x_1 + x_3 + x_4 + x_5 \leq 3 \}.
\end{align*}
}
The inequality~\eqref{eq:ex3} is indeed a CG cut with respect to $K_{CI}$, as shown by multipliers $\frac{1}{12} \cdot (19, 11, 5,4,2) + \frac{1}{4} \cdot (1,1,1,0,0) + \frac{1}{3} \cdot (1,1,0,1,0) + \frac{1}{2} \cdot (1,1,0,0,1) + \frac{1}{3} \cdot (1,0,1,1,1) = (3,2,1,1,1), \frac{1}{12} \cdot 31 + \frac{1}{4} \cdot 2 + \frac{1}{3} \cdot 2 + \frac{1}{2} \cdot 2 + \frac{1}{3} \cdot 3 = 5.75.$ Hence $(3,2,1,1,1)^T x \leq \lfloor 5.75 \rfloor = 5$ is a CG cut for $K_{CI}$. 
$\hfill\diamond$
\end{example}

Example~\ref{ex3} demonstrates that MCIs can be obtained from multiple knapsack sets simultaneously.  
Specifically, the inequality~\eqref{eq:ex3} is facet-defining for $\conv(K)$, but it is neither valid for $\{x \in \{0,1\}^5 \mid 19 x_1 + 11x_2 + 5x_3 + 4x_4 + 2x_5 \leq 31 \}$ nor for $\{x \in \{0,1\}^5 \mid 16 x_1 + 10 x_2 + 7 x_3 + 5 x_4 + 3 x_5 \leq 30\}$.  Example~\ref{ex3} also shows that an S-MCI can be a CG cut for the linear system given by the original knapsack constraints along with all their minimal cover inequalities.  
In the next example, we will see that this is not always the case.

\begin{example}
\label{ex4}
Consider the following TOMKS: 
{\small
\begin{align*}
K: = \{x \in \{0,1\}^8 \mid \ & 28 x_1 + 24x_2+20x_3+19x_4+15x_5+10x_6+7x_7+6x_8 \leq 96, \\
& 27x_1+24x_2+21x_3+19x_4+13x_5+12x_6+7x_7+4x_8 \leq 96\}.
\end{align*}
}
Define covers $C_1 = \{2,3,4,5,6,7,8\}$, $C_2 = \{1,3,4,5,6,8\},$ $C_3 = \{1,2,3,5,6\}$, $C_4 = \{1,2,3,5,7,8\}$.
We have $C = [8]$, $C_0 = \{3,5\}$, and the discrepancy family is $\mathcal{D}(\C) = \{\{2,4,6,7,8\}, \{1,4,6,8\}, \{1,2,6\}, \{1,2,7,8\}\} =: \{D_1, D_2, D_3, D_4\}$.

First, we verify that $\C$ is a multi-cover. For any set $T \subseteq C \setminus C_0$ and $T \notin \mathcal{D}(\C)$, if $1 \in T$, $|T| = 2$, then $T$ is clearly dominated by either $D_2, D_3$ or $D_4$. 
If $1 \in T$, $|T| = 3$, then either $T \dominates D_3$ or $D_3 \dominates T$. If $1 \in T$, $|T| = 4$, then $T$ must be comparable with $D_2$ or $D_3$.
If $1 \in T$, $|T| = 5$, then $T \dominates D_1$. If $1 \notin T,$ then clearly $D_1 \dominates T$ since $T \subseteq D_1$. Hence we have shown that for any $T \subseteq C \setminus C_0$ and $T \notin \mathcal{D}(\C)$, $T$ must be comparable with some set in $\mathcal{D}(\C)$. Therefore $\C$ is a multi-cover.

When  Algorithm~\ref{alg:mci} is applied to $\C$, we obtain the S-MCI $\alpha^T x \leq \beta$ given by
\begin{equation}
    4x_1+3x_2+3x_3+2x_4+3x_5+2x_6+x_7+x_8 \leq 14, \label{eq:ex4}
\end{equation}
and it can be shown that \eqref{eq:ex4} is facet-defining for $\conv(K)$. 

Consider the linear system given by all the minimal cover inequalities for $K$, as well as the original two linear constraints. We refer to this linear system as $K_{CI}$, which consists of 30 inequalities. 
Solving $\max\{\alpha^T x \mid x \in K_{CI}\}$ gives optimal value $15.307$, so the corresponding CG cut with respect to $K_{CI}$ with the same left-hand-side coefficient vector $\alpha$ is $\alpha^T x \leq 15$, which is weaker than inequality \eqref{eq:ex4}. 
$\hfill\diamond$
\end{example}

Even when the cover-family consists of covers all coming from the same knapsack inequality, the MCI can produce interesting inequalities. In the next example, we show an MCI that cannot be obtained as an LCI, regardless of the lifting order.

\begin{example}[Example 3 in \cite{letchford2019lifted}]
\label{exmp: 3}
Let $K := \{x \in \{0,1\}^5 \mid 10 x_1 + 7 x_2 + 7 x_3 + 4 x_4 + 4 x_5 \leq 16\}$, and consider the multi-cover $\C: = \{ \{1, 3\}, \{1, 4, 5\}, \{2, 3, 5\}\}$. 
From inequality \eqref{eq: CAI_3} of Example~\ref{exam: closed_2}, we know that the corresponding S-MCI is
\begin{equation}
    3x_1 + 2x_2 + 2x_3 + x_4 + x_5 \leq 4. \label{eq:ex5}
\end{equation}
The inequality~\eqref{eq:ex5} is the same inequality produced by the new lifting procedure described in \cite{letchford2019lifted}, and the authors of \cite{letchford2019lifted} state that \eqref{eq:ex5} is both a facet of $\conv(K)$ and cannot be obtained from any cover inequality by standard sequential lifting methods, regardless of the lifting order.
$\hfill \diamond$
\end{example}

Next, we discuss how the well-known \emph{$(1,k)$-configuration inequality} can be derived from the MCI.

\begin{proposition}
\label{prop: 1-k-AMCI}
Consider a knapsack set $K = \{x \in \{0,1\}^n \mid a^T x \leq b\}$, a nonempty subset $Q \subseteq [n]$, and $t \in [n] \setminus Q$. 
Assume that $\sum_{i \in Q} a_i \leq b$ and that $H \cup \{t\}$ is a minimal cover for all $H \subset Q$ with $|H| = k$.
Then for any $T(r) \subseteq Q$ with $|T(r)| = r$, $k \leq r \leq |Q|$, the $(1,k)$-configuration inequality
$$
(r-k+1) x_t + \sum_{j \in T(r)}x_j \leq r
$$
can be obtained from an MCI associated with two covers of $K$.
\end{proposition}

\begin{proof}
When $r = k$, then the inequality in the statement reduces to a CI. 
Hence we assume $r > k$. 
W.l.o.g. we assume $a_1  \geq \ldots \geq a_n$. Consider a new knapsack set $K': = \{x \in \{0,1\}^{n+1} \mid a^{\prime T} x \leq b\}$, with $a'_i = a_i \ \forall i \leq t$, $a'_{t+1} = a_t$, $a'_j = a_{j-1} \ \forall j > t+1.$ 
Then clearly we have $a'_1 \geq \ldots \geq a'_{n+1}$. 

Since for any $H \subset Q$ with $|H| = k$, $H \cup \{t\}$ is a cover for $K$, we know that for any $j \in Q$, the set $Q  \cup \{t\} \setminus \{j\}$ is also a cover for $K$, i.e., $\sum_{i \in Q} a_i - a_j + a_t > b$. 
From the assumption that $\sum_{i \in Q}a_i \leq b$, we have $a_t > a_j$, or equivalently, $t < j$ for any $j \in Q$. 
Now for any $T(r) \subseteq Q$ with $|T(r)| = r$, $k \leq r \leq |Q|$, let $T(r): = \{j_1, \ldots, j_r\}$ with $j_1 < \ldots < j_r$, so we have $t < j_1$ from above. 
Then consider $C_1: = \{t\} \cup \{j_{r-k+1}, \ldots, j_r\}$, $C'_1: = \{t\} \cup \{j_{r-k+1} + 1, \ldots, j_r + 1\}$, and $C'_2 : = \{t+1\} \cup \{j_{1} + 1, \ldots, j_r + 1\}$. 
Since $\{j_{r-k+1}, \ldots, j_r\} \subset Q$ with $|\{j_{r-k+1}, \ldots, j_r\}| = k$, we know that $C_1$ is a cover for $K$, 
so $C'_1$ is a cover for $K'$ from the construction of $K'$. 
Furthermore $C'_2$ is a cover for $K'$ since $a'_{t+1} = a'_t$. Note that the discrepancy family of $\{C'_1, C'_2\}$ is $\{\{t\}, \{t+1, j_1 + 1, \ldots, j_{r-k} + 1\}\}$, thus by Example~\ref{exam: 3}, the inequality
\begin{equation}
\label{eq: mci_1k}
(r-k+1) x_t + x_{t+1} + \sum_{\ell=1}^r x_{j_\ell + 1} \leq r
\end{equation}
is an MCI for $K'$ associated with $\{C'_1, C'_2\}$.
Since $K$ can be obtained by projecting out of the variable $x_{t+1}$ of $K'$, we project out the $x_{t+1}$ variable in \eqref{eq: mci_1k} and relabel the indices.
We obtain the following (1,k)-configuration inequality for $K$:
$$
(r-k+1) x_t + \sum_{\ell=1}^r x_{j_\ell} = (r-k+1) x_t + \sum_{j \in T(r)}x_j \leq r.
$$
\qed \end{proof}

\subsection{Extended MCI}
\label{subsec: extended_MCI}
In this section, we propose a procedure to strengthen, or extend, an MCI in a similar fashion to a well-known procedure for CIs \cite{balas1975facets}.  We call the strengthened inequalities \emph{extended MCI (E-MCI)}.

For any set $C \subseteq [n]$, let $\min(C)$ denote the least element in $C$.
Recall that for a cover $C \subseteq [n]$, its corresponding \emph{extended cover inequality} ECI is simply:
\begin{equation}
\label{eq: normal_CI}
x([\min(C)-1] \cup C) \leq |C| - 1,
\end{equation}
where the coefficient of each index $i$ that is less than $\min(C)$ is increased from 0 to 1. 
For a multi-cover $\{C_1, \ldots, C_k\}$ one can perform a similar extension: let $\alpha^T x \leq \beta$ be an MCI of this multi-cover, then the inequality $x([\min(C) - 1]) + \alpha^T x \leq \beta$ is also valid for $K$. 
However, the improved coefficients in general can be larger than 1, and the indices of the variables included in the inequality do not have to be limited in the set $[\min(C) - 1]$. 
Before introducing the formal definition for our extended inequality, for any $C_h$ and vector $\alpha,$ we denote by $\alpha_{2, C_h}$ the second least number in the multiset $\{\alpha_i, i \in C_h\}$, which allows duplication. 
In cases where multiple minimum numbers exist in $\{\alpha_i, i \in C_h\}$, then $\alpha_{2, C_h}$ is simply $\min_{i \in C_h} \alpha_i$. 

\begin{definition}[Extended MCI]
Given an MCI $\alpha^T x \leq \beta$ generated from a multi-cover $\{C_1, \ldots, C_k\}$. Then the following inequality
\begin{equation}
\label{eq: e-MCI}
\sum_{i \in [n] \setminus C} \big(\max_{h \in [k]: i < \min(C_h)} \alpha_{2, C_h} \big)  x_i + \alpha^T x \leq \beta
\end{equation}
is called an \emph{extended MCI} (E-MCI).
\end{definition}
Same as the remark we made after Algorithm~\ref{alg:mci}, here for an index $i$ with $\{h \in [k] \mid i < \min(C_h)\} = \emptyset$, the maximum over this empty set is set to be zero. 
It is easy to observe that, when \eqref{eq: e-MCI} is applied to the normal CI, this inequality  gives the same ECI as \eqref{eq: normal_CI}. 

Next, we prove that inequality \eqref{eq: e-MCI} is indeed a valid inequality for $K$.
\begin{theorem}
For a TOMKS $K$ and an MCI $\alpha^T x \leq \beta$ generated from a multi-cover $\{C_1, \ldots, C_k\}$, the E-MCI \eqref{eq: e-MCI} is valid for $K$. 
\end{theorem}

\begin{proof}
In order to show that inequality \eqref{eq: e-MCI} is valid for $K$, it suffices to show that for any $S \subseteq [n]$ whose incidence vector $\chi^S$ violates the inequality \eqref{eq: e-MCI}, then $S$ must be a cover for $K$. 
Let $E: = \{i \in [n] \setminus C \mid \exists h \in [k] \ \st \ i < \min(C_h)\}$.
The proof is by induction on $|S \cap E|$. 

When $|S \cap E| = 0$, $\chi^S$ violates inequality \eqref{eq: e-MCI} if and only if $\chi^S$ violates the original MCI $\alpha^T x \leq \beta$, which implies that $S$ is a cover for $K$ since we know that $\alpha^T x \leq \beta$ is valid for $K$. 
Now we assume that when $|S \cap E| = N (N<|E|)$ our statement is true. 
Consider the case where $|S \cap E| = N+1$.
We want to show that $\chi^S$ is not in $K$. 
Arbitrarily pick $i^* \in S \cap E$, with $\max_{h \in [k]: i^* < \min(C_h)} \alpha_{2, C_h} = \alpha_{2, C_{h^*}}$ for some $h^* \in [k]$.
We consider two cases.
\begin{enumerate}
\item If $\{i \in C_{h^*} \mid \alpha_i \geq \alpha_{2, C_{h^*}}\} \subseteq S$, by definition of $\alpha_{2, C_{h^*}}$, we know that $\{i \in C_{h^*} \mid \alpha_i < \alpha_{2, C_{h^*}}\}$ has at most one element, and that element is larger than $i^*$ since $i^* < \min(C_h^*)$. Therefore, we obtain 
\begin{align*}
S & \supseteq \{i^*\} \cup \{i \in C_{h^*} \mid \alpha_i \geq \alpha_{2, C_{h^*}}\} \\
&  \dominates \{i \in C_{h^*} \mid \alpha_i < \alpha_{2, C_{h^*}}\} \cup \{i \in C_{h^*} \mid \alpha_i \geq\alpha_{2, C_{h^*}}\} \\
&= C_{h^*},
\end{align*}
which indicates that $S$ must also be a cover for $K$. 
\item 
If there exists $j^* \in \{i \in C_{h^*} \mid \alpha_i \geq \alpha_{2, C_{h^*}}\}$ and $j^* \notin S$, then we have $\alpha_{j^*} \geq \alpha_{2, C_{h^*}}$, which is the coefficient of $x_{i^*}$ in equality \eqref{eq: e-MCI}.
Consider $S': = S \setminus \{i^*\} \cup \{j^*\}$, which is dominated by $S$ since $i^* < \min(C_h^*)$. 
Since $\chi^S$ violates inequality \eqref{eq: e-MCI}, we know that $\chi^{S'}$ also violates inequality \eqref{eq: e-MCI}. 
Note that $|S' \cap E| = |S \cap E| - 1 = N$, so by the inductive hypothesis, we know that $S'$ must be a cover for $K$. 
Since $S \dominates S'$, we obtain that $S$ must also be a cover for $K$.
\end{enumerate}
By induction, we conclude the proof.
\qed \end{proof}

The next example is obtained from Example~\ref{exmp: 3} by simply adding two more variables.
\begin{example}
Let $K: = \{x \in \{0,1\}^7 \mid 10 x_1 + 10 x_2 + 7 x_3 + 7 x_4 + 7 x_5 + 4 x_6 + 4 x_7 \leq 16\}$, and we consider the multi-cover $\C: = \{ \{2, 5\}, \{2, 6, 7\}, \{4, 5, 7\}\}$. From Example~\ref{exam: closed_2}, we have the following S-MCI: 
$$
3x_2 + 2x_4 + 2x_5 + x_6 + x_7 \leq 4.
$$
Now we lift the coefficients of indices $1$ and $3$. For index $3$, note that $\{h \in [3] \mid 3 < \min(C_h)\} = \{3\}$,
and the second smallest number in $\{\alpha_4, \alpha_5, \alpha_7\}$ is 2,
so the extended coefficient for $x_3$ is 2.
Similarly, for index $1$ we have $\{h \in [3] \mid 1 < \min(C_h)\} = \{1,2,3\},$ and $\alpha_{2, C_1} = 3, \alpha_{2, C_2} = 1, \alpha_{2, C_3} = 2$, so the extended coefficient for $x_1$ is 3. Therefore, we obtain the following E-MCI:
\begin{equation}
3x_1 + 3x_2 + 2x_3 + 2x_4 + 2x_5 + x_6 + x_7 \leq 4.
\end{equation}
In fact, this inequality turns out to be the only non-trivial facet-defining inequality of $\conv(K)$ that is not an ECI. 
$\hfill\diamond$
\end{example}

We give another interesting example with only one knapsack constraint.
\begin{example}
Let $K: = \{x \in \{0,1\}^6 \mid 66 x_1 + 61 x_2 + 54 x_3 + 33 x_4 + 21 x_5 + 16 x_6 \leq 130\}$. 
From the multi-cover $\{C_1, C_2\}: = \{\{2,3,6\},\{2,4,5,6\}\}$ we obtain an MCI: $3 x_2 + 2x_3 + x_4 + x_5 + x_6 \leq 5$. 
Note that $\alpha_{2, C_1} = 2$, $\alpha_{2, C_2} = 1$, so from \eqref{eq: e-MCI} we obtain the corresponding E-MCI $2x_1 + 3 x_2 + 2x_3 + x_4 + x_5 + x_6 \leq 5$, which is also a facet-defining inequality for $\conv(K)$. 
$\hfill\diamond$
\end{example}

\section{Facet-defining Inequalities}
\label{sec:facet}

In this section, we provide a sufficient condition for a S-MCI to define a facet of $\conv(K)$. We also given a family of instances in which all the non-trivial facet-defining inequalities of $\conv(K)$ are given by MCIs.

Given a multi-cover $\{C_1,\dots,C_k\}$ and an S-MCI $\alpha^T x \leq \beta$, we denote by $\{i_{t,1}, \ldots, i_{t, n_t}\}: = \{i \in C \setminus C_0  \mid \alpha_i = t\}$ the set of indices whose S-MCI coefficients are all $t$, where $i_{t, 1} < \ldots < i_{t, n_t}$.

\begin{theorem}
\label{theo: facet_condition_1}
Let $\{C_1,\dots,C_k\}$ be a multi-cover for a TOMKS $K$, and let $\alpha^T x \leq \beta$ be the corresponding S-MCI. 
Assume that the following conditions hold:
\begin{enumerate}
\item $C_0 = \emptyset$; \label{cond: 1}
\item For each $h \in [k]$, cover $C_h$ is a minimal cover; \label{cond: 2}
\item For any $t = 2, \ldots, \max_{i=1}^n \alpha_i$, there exists some $ i_{t-1, \ell_t} \notin C_{h_t} \in  \{C_h\}_{h=1}^k$ with
$i_{t,1} \in C_{h_t}$ and $i_{1,n_1} \in C_{h_t}$, such that
$C_{h_t} \cup \{i_{t-1, \ell_t}\} \setminus \{i_{t, n_t}\}$ is not a cover; \label{cond: 3}
\item There exists some $C_{h_1} \in \{C_h\}_{h=1}^k$, such that $i_{1, 1} \in C_{h_1}$ and for any $i' \notin C$, $C_{h_1} \cup \{i'\} \setminus \{i_{1, 1}\}$ is not a cover. \label{cond: 4}
\item For any $t = 1, \ldots, \max_{i=1}^n \alpha_i$, $\alpha(C_{h_t}) = \beta + 1$. \label{cond: 5}
\end{enumerate}
Then $\alpha^T x \leq \beta$ is a facet-defining inequality of $\conv(K)$.
\end{theorem}

The proof of this theorem is deferred to Appendix~\ref{subsec: proof_facet_condition_1}.

\begin{example}
Consider the TOMKS and the multi-cover in Example~\ref{exmp: 3}.
We have $C_1 = \{1, 3\}$, $C_2 = \{1, 4, 5\}$, $C_3 = \{2, 3, 5\}$, Then the corresponding S-MCI $\alpha^T x \leq \beta$ is $3x_1 + 2x_2 + 2x_3 + x_4 + x_5 \leq 4$. Here we have $i_{1,1} = 4, i_{1,2} = 5, i_{2,1} = 2, i_{2,2} = 3, i_{3,1} = 1$. 

Clearly condition~\ref{cond: 1} in Theorem~\ref{theo: facet_condition_1} holds.
Since $\alpha(C_1) - \alpha_3 = 10 \leq 16$, $\alpha(C_2) - \alpha_5 = 14 \leq 16$, $\alpha(C_3) - \alpha_5 = 14 \leq 16$, condition~\ref{cond: 2} holds as well.
For $t = 2,$ let $C_{h_2} = C_3$, then $i_{1,1} \notin C_{h_2}, i_{1,2} \in C_{h_2}, i_{2,1} \in C_{h_2}$, and $C_{h_2} \cup \{i_{1,1}\} \setminus \{i_{2,2}\} = \{2, 4, 5\}$ is not a cover. For $t = 3$, let $C_{h_3} = C_2$, then $i_{2,1} \notin C_{h_3}, i_{1,2} \in C_{h_3}, i_{3,1} \in C_{h_3}$, and $C_{h_3} \cup \{i_{2,1}\} \setminus \{i_{3,1}\} = \{2,4,5\}$ is not a cover. So condition~\ref{cond: 3} holds.
Let $C_{h_1} = C_2$, then $i_{1,1} \in C_{h_1}$, since here $C = [5]$, condition~\ref{cond: 4} holds.
Lastly, $\alpha(C_{h_1}) = \alpha(C_{h_2}) = \alpha(C_{h_3}) = 5$, so condition~\ref{cond: 5} also holds. 
Hence Theorem~\ref{theo: facet_condition_1} yields that this S-MCI is facet-defining.
$\hfill \diamond$
\end{example}

A \emph{clutter} is a family of subsets of a ground set with the property that no subset in the clutter is contained in any other subset in the clutter. 
It is simple to check that the set of minimal covers $\mathcal C$ of a knapsack set is a clutter, which we call the \emph{minimal cover set}. 
For a clutter $\mathcal C$ and every subset $Z$ of $[n]$, the \emph{deletion}
is defined as $\{A \in \mathcal C \mid A \cap Z =\emptyset\}$, and the \emph{contraction}
is defined as $\{A-Z \mid A \in \mathcal C\}$. 
A \emph{minor} of $\mathcal C$ is a clutter that may be obtained from $\mathcal C$ by a sequence of deletions and contractions. 
For a minimal cover set of a TOMKS, 
we have the following theorem that was mentioned in Section~\ref{sec:intro}.

\begin{theorem}
\label{theo: cover_clutter}
Let $\mathcal C$ be the minimal cover set of a TOMKS $K$. Then $\conv(K) = \{x \in [0,1]^n \mid x(C) \leq |C|-1, \forall C \in \mathcal C\}$ if and only if $\mathcal C$ has no minor isomorphic to $J_q = \{\{2,\ldots, q\}, \{1,i\} \text{ for }  i = 2, \ldots, q\}$ with $ q \geq 3.$
\end{theorem}

Theorem~\ref{theo: cover_clutter} follows easily from Theorem 1.1 in \cite{laurent1992characterization} (see also \cite{seymour1977matroids}) and the fact that the minimal cover set $\mathcal C$ of a TOMKS $K$ is indeed a \emph{shift clutter} as defined in \cite{bertolazzi1987mn}. 
For completeness, we provide a proof in Appendix~\ref{subsec: proof_cover_clutter}.

From Theorem~\ref{theo: cover_clutter}, we know that in order for $\conv(K)$ to have non-trivial facet-defining inequalities other than CIs, the minimal cover set $\mathcal C$ must have minors isomorphic to $J_q$. 
Now we consider a special clutter $\{\{1, \ldots, p-1, p+1, \ldots, q\}, \{1, \ldots, p, i\} \text{ for }i=p+1, \ldots, q\}$ for some $p = 1, \ldots, q-2, q \leq n$. 
After contracting $\{1, \ldots, p-1\}$, we obtain a minor $\{\{p+1, \ldots, q\}, \{p, i\} \text{ for } i = p+1, \ldots, q\}$ that is isomorphic to $J_{q-p+1}$. 
The next theorem states that, if the minimal cover set is this particular clutter, then we can provide the complete linear description of $\conv(K)$.

\begin{theorem}
\label{prop: suff_full_3}
Let $K$ be a TOMKS whose minimal cover set is \\
$\{\{1, \ldots, p-1, p+1, \ldots, q\}, \{1, \ldots, p, i\} \text{ for }i=p+1, \ldots, q\}$ for some $p = 1, \ldots, q-2$, $q \leq n$. 
Then $\conv(K)$ can be described by the bound constraints, the CIs: $\sum_{i=1}^p x_i + x_j \leq p \text{ for } j=p+1, \ldots, q,$ $\sum_{i=1}^{p-1} x_i + \sum_{i=p+1}^q x_i \leq q-2$, and one MCI: $(q-p) \sum_{i=1}^{p-1} x_i + (q-p-1)x_p + \sum_{i=p+1}^q x_i \leq p(q-p)-1$.
\end{theorem}

Since CI is a special case of MCI, Theorem~\ref{prop: suff_full_3} can be seen as a particular instance where the facet-defining inequalities are all given by MCIs.
The proof of this theorem can be found in Appendix~\ref{subsec: proof_suff_full_3}.

\section{Separation Problem}
\label{sec: separation}

 The \emph{separation problem} for a multiple knapsack set $K$ is: ``Given a vector $\tilde x \in [0,1]^n$, either find an inequality that is valid for $K$ and violated by $\tilde x$, or prove that no such inequality exists''. 
Even though the separation problems for some well-known classes of inequalities, e.g., CIs, simple LCIs and general LCIs have all been shown to be $\mathcal{NP}$-hard \cite{ferreira1996solving,gu1999lifted,klabjan1998complexity}, several efficient heuristics and exact separation algorithms have been proposed \cite{crowder1983solving,gu1998lifted,wolsey1999integer,weismantel19970}. 
In this section, we propose a \emph{mixed-integer programming (MIP)} formulation to solve the exact separation problem for MCIs.

First, we introduce the following concept of a \emph{skeleton}.
For $D \subseteq [n]$ and a function $f: D \to \N$, we denote by $f(D): = \{f(i) \mid i \in D\}$ the image of $D$ under $f$. 
 \begin{definition}[Skeleton]
 Let $\C \subseteq 2^{[n]}$.
 For any $i \in \cup_{D \in \mathcal{D}(\C)} D$, let $f(i) := k$ if $i$ is the $k$-th least element of $\cup_{D \in \mathcal{D}(\C)} D$. 
 Then we say that $\{f(D) \mid D \in \mathcal{D}(\C)\}$ is the \emph{skeleton} of $\C$.
 \end{definition}
In other words, the skeleton of $\C$ is isomorphic to $\mathcal{D}(\C)$ by relabelling the elements of $\cup_{D \in \mathcal{D}(\C)} D$ with their order in the set. 


Due to the lack of a complete understanding of multi-covers, it is currently challenging to derive an efficient scheme to separate various MCIs altogether. 
In fact, for any given family of covers, we conjecture that it is $\mathcal{NP}$-hard to decide whether or not it is a multi-cover.
However, different multi-covers with the same (or similar) skeleton structure tend to have very similar combinatorial properties.
It turns out that for any fixed skeleton, we are able to formulate as a MIP problem the separation problem for MCIs whose associated multi-covers have that skeleton. 
We present the exact formulation in Appendix~\ref{subsec: additional_formulation_sep}. Here we mainly introduce a different formulation, which enables us to separate a sub-family of MCIs whose associated multi-cover's skeleton is either $\{\{1\}, \{2, \ldots, t\}\}$ or $\{\{1,t+1\}, \{2, \ldots, t\}\}$, for any $t \geq 3$.  There are the inequalities given in Examples~\ref{exam: 3} and \ref{exam: 1_t+1}. 
Given a vector $\tilde x \in [0,1]^n$, we provide below the formulation for the corresponding separation problem:


\begin{align}
\label{Sep2}
\tag{Sep2}
\begin{split}
\min \quad & t + \sum_{i \in [n]} \gamma_i - \sum_{i \in [n]} (\alpha_i + \beta_i + \gamma_i) \tilde x_i \\
\st \quad & t \geq \sum_{i \in [n]} \alpha_i, \ t \geq  \sum_{i \in [n]} \beta_i, \\
& u_i + v_i + w_i \leq 1, \quad \forall i \in [n] \\
&  \alpha_i \geq \beta_j + (1+M)u_i - M, \ \beta_i \geq \alpha_j + (1+M) v_i - M, \quad  \forall i,j \in [n], \ j > i\\
& w_{1,i} + w_{2,i} = w_i, \quad \forall i \in [n] \\
&  \alpha_i \leq M u_i, \ \beta_i \leq M v_i, \ \gamma_i \leq M w_i, \quad \forall i \in [n] \\
& \gamma_i \geq \sum_{j>i} \alpha_j + (1+nM)w_{1,i} -nM, \ \gamma_i \geq \sum_{j>i} \beta_j + (1+nM) w_{2,i} -nM, \quad \forall i \in [n]\\
& \gamma_i \geq  \alpha_j + Mw_{1,i}-M, \ \gamma_i \geq  \beta_j + M w_{2,i} - M, \quad \forall i,j \in [n], \ j < i \\
& \sum_{i \in [n]} A_{j,i} (u_i + w_i) \geq (b_j+1) \cdot \lambda_j, \ \sum_{i \in [n]} A_{j,i} (v_i + w_i) \geq (b_j+1) \cdot \mu_j, \quad \forall j \in [m] \\
& \sum_{j \in[m]} \lambda_j \geq 1, \sum_{j \in[m]} \mu_j \geq 1,\\
& \sum_{i \in [n]} v_i \geq 2, \ \sum_{j<i} u_j \geq v_i, \ v_i + \sum_{j \leq i} u_j \leq 2, \quad \forall i \in [n] \\
& \lambda_j, \mu_j \in \{0,1\}, \quad \forall j \in [m] \\
& u_i, v_i, w_i, w_{1,i}, w_{2,i}, z_i \in \{0,1\}, \quad \forall i \in [n] \\
& \alpha_i, \beta_i, \gamma_i \in \Z, \quad \forall i \in [n].
\end{split}
\end{align}

Next we show that, the separation problem for MCIs from multi-covers with skeleton $\{\{1\}, \{2, \ldots, t\}\}$ or $\{\{1,t+1\}, \{2, \ldots, t\}\}$ can be exactly solved by solving the above MIP problem.

\begin{theorem}
\label{theo: sep_2}
Given a TOMKS $K$ and a point $\tilde x$, there exists a multi-cover $\C$ whose skeleton is $\{\{1\}, \{2, \ldots, t\}\}$ or $\{\{1,t+1\}, \{2, \ldots, t\}\}$ for some $t \geq 3$ and an associated MCI to separate $\tilde x$ from $K$, if and only if \eqref{Sep2} has an optimal value less than 1. 
\end{theorem}

\begin{proof}
First, we explain the meaning of the variables in \eqref{Sep2}. 
Let $C_1$ and $C_2$ be two covers, for any $i \in [n]$, we define $u_i := \mathbbm{1}\{i \in C_1 \setminus C_2\}$, $v_i := \mathbbm{1}\{i \in C_2 \setminus C_1\}$, and $w_i: = \mathbbm{1}\{i \in C_1 \cap C_2\}$. 
So we have the second constraint $u_i + v_i + w_i \leq 1$ in \eqref{Sep2}. 
Variables $\alpha_i, \beta_i, \gamma_i$ denote the coefficients for index $i$ in $C_1 \setminus C_2, C_2 \setminus C_1$ and $C_1 \cap C_2$ respectively, which is enforced by $\alpha_i \leq M u_i$ etc., through some ``big-M" constant. 
From Step~7 of Algorithm~\ref{alg:mci}, 
$$
\sum_{i \in [n]} (\alpha_i + \beta_i + \gamma_i) x \leq \max(\sum_{i} \alpha_i + \sum_{i} \gamma_i, \sum_{i} \beta_i + \sum_{i}\gamma_i) - 1
$$ 
is the MCI we will obtain. 
Introducing an additional variable $t$ and constraints $t \geq \sum_i \alpha_i$, $t \geq \sum_i \beta_i$, we obtain the objective function and the first constraint of \eqref{Sep2}. 
In particular, the objective value is strictly less than 1 if and only if the obtained inequality is violated by $\tilde x$. 
Using a big-$M$ formulation, $\alpha_i \geq \beta_j + (1+M)u_i - M$ formulates the constraint of Step~4 of Algorithm~\ref{alg:mci}: for $i \in C_1 \setminus C_2, \alpha_i \geq \max_{j >i, j \in C_2 \setminus C_2} \beta_j + 1$.
By splitting the binary variable $w_i$ into two binary variables $w_{1,i}$ and $w_{2,i}$, and using the big-$M$ formulation
$\gamma_i \geq \sum_{j>i} \alpha_j + (1+nM)w_{1,i} -nM, \gamma_i \geq  \alpha_j + Mw_{1,i}-M$ etc., we are able to formulate the constraint of Step~6 of Algorithm~\ref{alg:mci}: 
\begin{align*}
\begin{split}
& \gamma_i \geq \max\{\max_{j<i, j \in C_1 \setminus C_2} \alpha_j, \sum_{j>i, j \in C_1 \setminus C_2} \alpha_j + 1\} \text{ or }  \\
& \gamma_i \geq \max\{\max_{j<i, j \in C_2 \setminus C_1} \beta_j, \sum_{j>i, j \in C_2 \setminus C_1} \beta_j + 1\}.
\end{split}
\end{align*}

Constraints $\sum_{i \in [n]} A_{j,i} (u_i + w_i) \geq (b_j+1) \cdot \lambda_j, \ \sum_{i \in [n]} A_{j,i} (v_i + w_i) \geq (b_j+1) \cdot \mu_j$ and $\sum_i \lambda \geq 1, \sum_i \mu_i \geq 1$ enforce that $C_1$ and $C_2$ are two covers of $K$. 
Lastly, the constraints $\sum_{i \in [n]} v_i \geq 2, \sum_{j<i} u_j \geq v_i, v_i + \sum_{j \leq i} u_j \leq 2$ enforce the $\{\{1\}, \{2, \ldots, t\}\}$ or $\{\{1,t+1\}, \{2, \ldots, t\}\}$ skeleton structures with $t \geq 3$.
In particular, constraint $\sum_{j<i} u_j \geq v_i$ means $\min\{i \mid i \in C_1 \setminus C_2\} < \min\{i \mid i \in C_2 \setminus C_1\}$, and constraint $v_i + \sum_{j \leq i} u_j \leq 2$ means $|C_1 \setminus C_2| \leq 2$, and $\max\{i \mid i \in C_1 \setminus C_2\} > \max\{i \mid i \in C_2 \setminus C_1\}$ if $|C_1 \setminus C_2| = 2$. These are exactly the conditions for $\{C_1, C_2\}$ to have skeleton $\{\{1\}, \{2, \ldots, t\}\}$ or $\{\{1,t+1\}, \{2, \ldots, t\}\}$. 
\qed \end{proof}

It's worth mentioning that, the constraint $\sum_{i \in [n]} v_i \geq 2$ in \eqref{Sep2} can actually be removed. In that case, \eqref{Sep2} will also be able to separate MCIs with skeleton $\{\{1\}, \{2, \ldots, t\}\}$ or $\{\{1,t+1\}, \{2, \ldots, t\}\}$ for some $t \leq 2$, as well as the normal CIs, since the optimal solution to \eqref{Sep2} with binary variables $u=v=0$ corresponds to a separating CI, where the cover is given by the support of the $w$ vector.

\section{Numerical Experiments}
\label{sec: numerical}

In this section we present results of numerical experiments designed to test the optimality gap closed by our proposed cutting-planes and their variants.

For a given fractional solution, we use the separation formulation \eqref{Sep2} to produce MCIs that arise from multi-covers with two specific skeletons. We also relax the constraint $\sum_{i \in [n]} v_i \geq 2$ therein, so the separation problem may separate CI as well.
The numerical experiment proceeds as follows: for a given totally-ordered multiple knapsack problem, in the $i$-th iteration we optimize the linear objective function over the current linear relaxation, and obtain an optimal solution $x^i$. We terminate the experiment if the separation problem \eqref{Sep2} associated with $x^i$ does not have optimal value being strictly less than 1. Otherwise, we add the separating cutting-plane generated from \eqref{Sep2} into the current linear relaxation, then proceed to the next step. The linear relaxation in the first iteration is initialized as  the natural LP relaxation of the problem. 

For our experiments, we create synthetic instances of the multiple knapsack problem: $\max\{c^T x \mid Ax \leq b, x \in \{0,1\}^n\}$, where $A \in \Z^{m \times n}_+, c \in \Z^n_+, b \in \Z^m_+$.
For each row of the matrix $A_j$, we generate $n$ random integers in the range $[1, n^2]$ and then sort them in non-increasing order.
The right-hand-side number $b_j$ is generated as a random integer number in the range $[A_{j,1}, \sum_i (A_{j,i})]$.
The objective vector $c$ is also generated from sorting $n$ random integer numbers in $[1, n^2]$ in non-increasing order. 
We create ten instances of sizes $n \in \{20,30\}$ and with $m \in \{1,2,3\}$ constraints.

We compare the MCI inequality and two of its variants against analogous CIs.  Specifically, we will compare the MCI against the original CI, the E-MCI against the classical ECI, and a lifted version of MCI \emph{(L-MCI)} with a lifted CI (\emph{LCI}).
To obtain a lifted CI or a lifted MCI, we start with the original CI or MCI, and then apply the simple sequential up-lifting procedure of \cite{padberg1975note}, lifting the variables in the order $1,2,\ldots,n$.  In subsequent tables and figures, we use the following abbreviations:

\begin{itemize}
    \item[] \emph{LP:} \quad Denotes the natural LP relaxation.
    \item[] \emph{MCI}: \quad Denotes the linear relaxation after iteratively solving the separation problem~\eqref{Sep2} to add all CIs and MCIs whose associated multi-covers have skeleton $\{\{1\}, \{2, \ldots, t\}\}$ or $\{\{1,t+1\}, \{2, \ldots, t\}\}$.
    \item[] \emph{E-MCI}: \quad Denotes the linear relaxation where each MCI found is strengthened/extended as in \eqref{eq: e-MCI}.
    \item[] \emph{L-MCI}: Denotes the linear relaxation where each MCI is lifted.
    \item[]  \emph{CI}: \quad Denotes the linear relaxation obtained after adding all CIs. Here the separation problem is exactly solved by solving an IP, as in  \cite{crowder1983solving}.
    \item[]  \emph{ECI}: \quad Denotes the linear relaxation where each additional inequality is the ECI of the separating CI in the current iteration.
    \item[]  \emph{LCI}: \quad Denotes the linear relaxation where each additional inequality is the LCI of the separating CI in the current iteration.  
\end{itemize}

We report the summary of numerical results in Table~\ref{smalltable1} and Table~\ref{smalltable2}, and more details can be found in Appendix~\ref{subsec: detailed_data}.
For different combinations of $(n,m)$,
in Table~\ref{smalltable1}, we list the average optimality gap obtained from optimizing over the different relaxations studied, and in Table~\ref{smalltable2}, we list the total number of instances that have been solved to optimality for each method.
From these two tables, one can see quite clearly that the \emph{L-MCI} is able to close much more optimality gap than its \emph{LCI} counterpart, and solve most ($60-80\%$) of the instances to optimality.  The difference between the two methods on each instance is shown directly in  Figure~\ref{fig:lmci_lci}, where we see that instances with over 3\% optimality gap using LCI may be solved at the root node to optimality with LMCI.

\begin{table}
\centering
\caption{Average optimality gap of different cutting-plane closure. (\%)}
\begin{tabular}{  p{1cm} | p{1cm} | p{1cm} p{1cm} p{1cm} | p{1cm} p{1cm} p{1cm} }
\hline
(n,m) & LP   & MCI & E-MCI & L-MCI &  CI & ECI & LCI \\ 
\midrule
(20, 1) &  3.93  & 1.08  & 0.27  & 0.11  & 2.00  & 0.65  &  0.63  \\
(20, 2) & 6.01   & 2.90  & 1.02  &  0.21 & 4.23  & 1.69  & 1.12   \\
(20, 3) &  6.03  & 3.89  & 0.73  &  0.27 &  4.67 & 1.17  & 1.12  \\
\midrule
(30, 1) & 6.34  & 4.78 & 0.59  &0.24  &  5.24 & 0.68  & 0.65 \\ 
(30, 2) & 4.76  & 2.95 & 0.89  &0.19  &  3.51 & 1.09  & 0.83 \\
(30, 3) & 3.13  & 2.50 & 0.83  &0.31  &  2.67 & 0.90  & 0.80 \\
\bottomrule
 \end{tabular}
 \label{smalltable1}
\end{table}
\begin{table}
\centering
\caption{Number of instances been solved to optimality. (out of 10)}
\begin{tabular}{  p{1cm} | p{1cm} | p{1cm} p{1cm} p{1cm} | p{1cm} p{1cm} p{1cm} }
\hline
(n,m) & LP  & MCI & E-MCI & L-MCI &  CI & ECI & LCI \\
\midrule
(20, 1) &  0   & 3   & 4  &  6  & 1   & 1   &  2  \\
(20, 2) &   0  & 0   & 2  &   6 & 0   &  1  &  2  \\
(20, 3)  & 0  &  1  & 4   & 8  &  0  &  2  &  2 \\
\midrule
(30, 1) & 0  & 1 & 3  & 7  &  1 & 3  & 3 \\ 
(30, 2) & 0  & 1 & 5  & 7  &  0 & 4  & 5 \\
(30, 3) & 0  & 1 & 2  & 6  &  1 & 2  & 2 \\
\bottomrule
 \end{tabular}
 \label{smalltable2}
\end{table}

\begin{figure}
    \centering
    \includegraphics[scale=0.6]{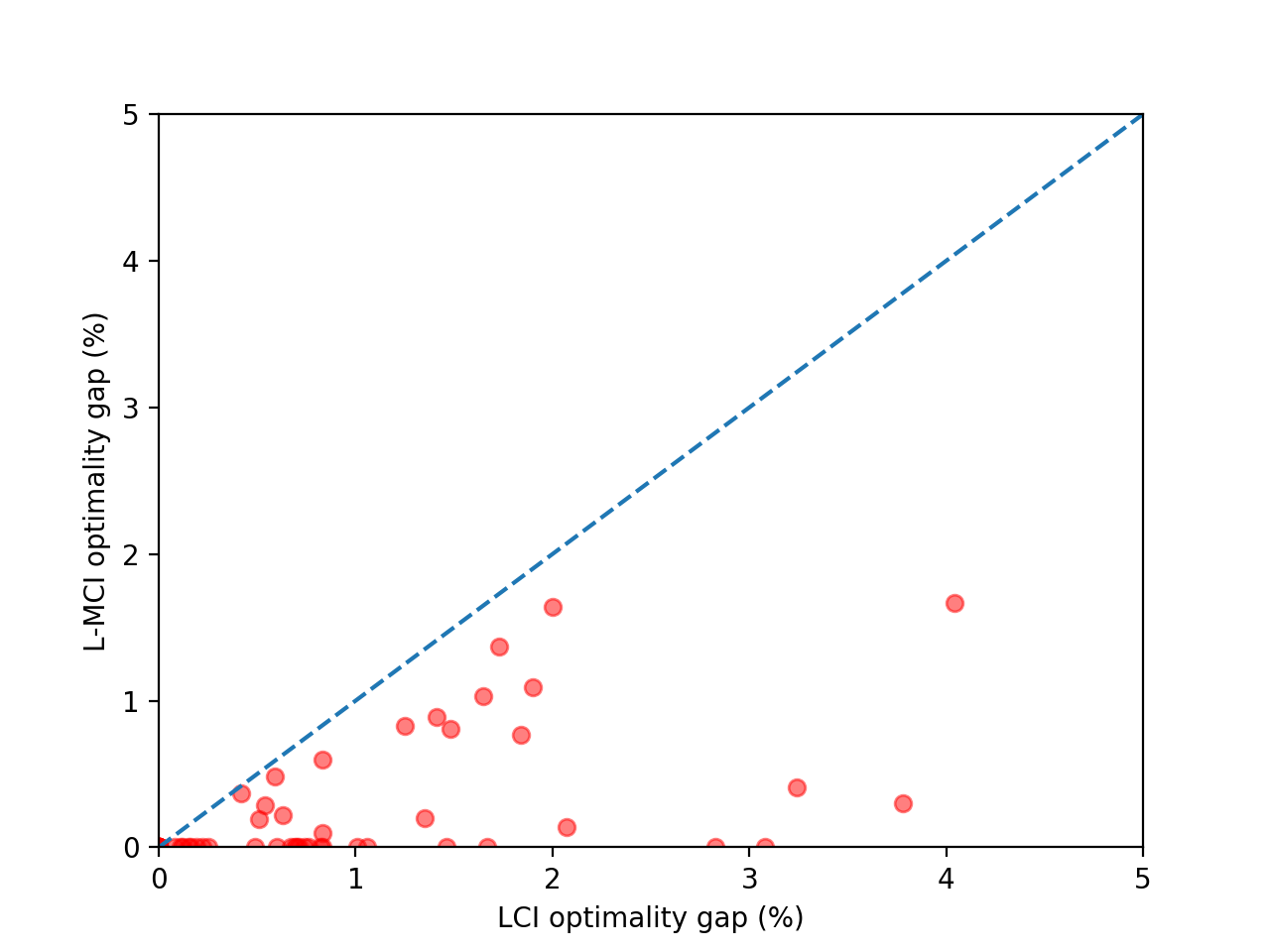}
    \caption{The x-axis and y-axis of each dot represent the optimality gap closed by LCI and L-MCI for each instance respectively.}
    \label{fig:lmci_lci}
\end{figure}


Because our separation method is based on the solution of a difficult MIP problem \eqref{Sep2}, we only conduct our experiments for small-sized knapsack instances.  However, we believe the computational results are promising enough to encourage others to seek efficient heuristic separation methods for this large family of inequalities.

\section{Conclusion}
In this work, we introduced a new family of valid inequalities for the intersection of knapsack sets and exhibited several scenarios in which the inequalities are not implied by other known families of cutting-planes. The numerical experiments demonstrated the potential for this new family of cuts to strengthen the linear programming relaxation more than known classes of inequalities.
Our work is among the very first that explicitly studies the polyhedral structure of the intersection of multiple knapsack sets, and we hope that the ideas presented here will give rise to new methods for generating strong valid inequalities for complex binary sets that arise in practical settings.

\bibliographystyle{spmpsci} 
\bibliography{biblio} 

\appendix
\addcontentsline{toc}{section}{Appendices}
\renewcommand{\thesubsection}{\Alph{subsection}}

\begin{appendices}

\section{Proof of Theorem~\ref{theo: facet_condition_1}}
\label{subsec: proof_facet_condition_1}

First, by Step~\ref{step: MCI_4} in Algorithm~\ref{alg:mci}, we have the following easy observation.
 \begin{observation}
 \label{obs: obvious}
 Let $\{C_1,\dots,C_k\}$ be a multi-cover and let $\alpha^T x \leq \beta$ its corresponding S-MCI. 
 If there exist some $t \in \N, \ell \in [n_t]$ and $h' \in [k]$ such that $i_{t, \ell} \in C_{h'}$, then $\{ i_{t,\ell}, \ldots, i_{t, n_t}\} \subseteq C_{h'}$.
 \end{observation}

Now we can prove Theorem~\ref{theo: facet_condition_1}.

\begin{proof}[Proof of Theorem~\ref{theo: facet_condition_1}]
Consider the set of binary points whose support is in one of the following sets:
\begin{align}
& \mathscr{S}_1: = \{C_{h_t} \cup \{i_{t-1, \ell_t}\} \setminus \{i_{t, \ell}\} \mid t = 2, \ldots, \max_{i=1}^n \alpha_i, \ell = 1, \ldots, n_t \}, \label{set: 1} \\
& \mathscr{S}_2: = \{C_{h_t} \setminus \{i_{1,n_1}\} \mid t = 2, \ldots,  \max_{i=1}^n \alpha_i\},  \label{set: 2}\\
& \mathscr{S}_3 := \{C_{h_1} \setminus \{i_{1, \ell}\} \mid \ell = 1, \ldots, n_1\},  \label{set: 3}\\
& \mathscr{S}_4 := \{C_{h_1} \cup \{i'\} \setminus \{i_{1, 1}\}  \mid i' \notin C\}.  \label{set: 4}
\end{align}

First, we want to prove that any set in \eqref{set: 1}-\eqref{set: 4} is not a cover for $K$. 
By condition~\ref{cond: 3}, we know $i_{1,n_1} \in C_{h_t}$ for any $t = 2, \ldots,  \max_{i=1}^n \alpha_i$, and by condition~\ref{cond: 4}, we know $i_{1,1} \in C_{h_1}$. 
From Observation~\ref{obs: obvious}, there is $\{i_{1,1}, \ldots, i_{1, n_1}\} \subseteq C_{h_1}$.
Hence by condition~\ref{cond: 2}, we know that for any $t = 2, \ldots,  \max_{i=1}^n \alpha_i, C_{h_t} \setminus \{i_{1,n_1}\}$ is not a cover for $K$, and for any $\ell = 1, \ldots, n_1, C_{h_1} \setminus \{i_{1,\ell}\}$ is not a cover for $K$. So any set in $\mathscr{S}_2 \cup \mathscr{S}_3$ is not a cover. Condition~\ref{cond: 4} directly states that any set in $\mathscr{S}_4$ is not a cover. Furthermore, condition~\ref{cond: 3} states that $i_{t,1} \in C_{h_t}$, then by Observation~\ref{obs: obvious},
we know that for any $\ell \in [n_t], i_{t, \ell} \in C_{h_t}$. Also, condition~\ref{cond: 3} states that $C_{h_t} \cup \{i_{t-1, \ell_t}\} \setminus \{i_{t, n_t}\}$ is not a cover, so $C_{h_t} \cup \{i_{t-1, \ell_t}\} \setminus \{i_{t,\ell}\}$ is also not a cover of $K$.
Hence any set in $\mathscr{S}_1$ is not a cover of $K$. 

Now we want to show that $\alpha^T x = \beta$ is the only hyperplane that contains all the incidence vectors of the sets in $\mathscr{S}_1 \cup \mathscr{S}_2 \cup \mathscr{S}_3 \cup \mathscr{S}_4$. 
Let $u^T x = v$ be the hyperplane that contains all those binary points. 
Then, from the sets in $\mathscr{S}_3$, we know that $u_{i_{1,1}} = \ldots = u_{i_{1,n_1}}$, and we denote it to be $\kappa$. Since for any $t = 2, \ldots, \max_{i=1}^n \alpha_i$ and $\ell = 1, \ldots, n_t, u(C_{h_t} \cup \{i_{t-1, \ell_t}\} \setminus \{i_{t, \ell}\}) = v$, we know that for any $t = 2, \ldots, \max_{i=1}^n \alpha_i, u_{i_{t,1}} = \ldots = u_{i_{t, n_t}}$. Furthermore, since $u(C_{h_t} \cup \{i_{t-1, \ell_t}\} \setminus \{i_{t, \ell}\}) = u(C_{h_t} \setminus \{i_{1,n_1}\}) = v$, we obtain that for any $t = 2, \ldots, \max_{i=1}^n \alpha_i$ and $\ell = 1, \ldots, n_t, u_{i_{t,\ell}} - u_{i_{t-1, \ell_t}} = u_{i_{1,n_1}} = \kappa$. 
Lastly, from the points in $\mathscr{S}_3$ and in $\mathscr{S}_4$, we know that $u_{i'} = 0$ for any $i' \notin C$. 
Hence we obtain that, for any $t = 1, \ldots, \max_{i=1}^n \alpha_i, \ell = 1, \ldots, n_t$, there is $u_{i_{t,\ell}} = \kappa \cdot t$, and for any $i' \notin C, u_{i'} = 0$. 
Since $\alpha_{i_{t,\ell}} = t$ and $\alpha_{i'} = 0$ for any $t = 1, \ldots, \max_{i=1}^n \alpha_i, \ell = 1, \ldots, n_t, i' \notin C$, 
and $\{i_{t, 1}, \ldots, i_{t, n_t}\} = \{i \in C \setminus C_0 \mid \alpha_i = t\}$,
we know that $u_i = \kappa \cdot \alpha_i$ for any $i \notin C_0$. 
By condition~\ref{cond: 1}, we have $u = \kappa \cdot \alpha$. 
Using condition~\ref{cond: 5}, it is simple to check that $v = u(C_{h_t}) - \kappa = \kappa \cdot \alpha(C_{h_t}) - \kappa = \kappa \cdot \beta$ for any $t = 1, \ldots, \max_{i=1}^n \alpha_i$.
Thus, we obtain that $(u,v) = \kappa \cdot (\alpha, \beta)$, and this concludes the proof that $\alpha^T x = \beta$ is the only hyperplane that contains all the incidence vectors of sets in $\mathscr{S}_1 \cup \mathscr{S}_2 \cup \mathscr{S}_3 \cup \mathscr{S}_4$, which are all in the TOMKS $K$. 
Since the S-MCI $\alpha^T x \leq \beta$ is a valid inequality for $\conv(K)$, we obtain that $\alpha^T x \leq \beta$ is a facet-defining inequality for $\conv(K)$. 
\qed \end{proof}

\section{Proof of Theorem~\ref{theo: cover_clutter}}
\label{subsec: proof_cover_clutter}

Here we restate Theorem~1.1 in \cite{laurent1992characterization}.
\begin{theorem}[Theorem 1.1 \cite{laurent1992characterization}]
\label{theo: seymour}
If a clutter $\mathscr{L}$ has no $P_4: = \{\{1,2\}, \{2,3\}, \{3,4\}\}$ minor, then 
$$\conv(\{x \in \Z^n_+ \mid x(L) \geq 1, \forall L \in \mathscr{L}\}) = \{x \in \R^n_+ \mid x(L) \geq 1, \forall L \in \mathscr{L}\}$$ 
if and only if $\mathscr{L}$ has no minor isomorphic to one of the following clutters:
\begin{enumerate}
\item $Q_6 = \{\{1,3,5\}, \{1,4,6\}, \{2,3,6\}, \{2,4,5\}\}$;
\item $J_q = \{\{2, \ldots, q\}, \{1, i\} \text{ for } i = 2, \ldots, q\}, q \geq 3$.
\end{enumerate}
\end{theorem}

\begin{proof}[Proof of Theorem~\ref{theo: cover_clutter}]
Let $\mathcal C$ be a minimal cover set of a TOMKS $K$.
We have the following claim.
\begin{claim}
\label{claim this one}
$\mathcal C$ has no minor isomorphic to $P_4$ or $Q_6$.
\end{claim}
\begin{cpf}
Let $\{\{i_1, i_2\}, \{i_2, i_3\}, \{i_3, i_4\}\}$ be the subset of one of the minors $M$ of $\mathcal C$, with $\{\{i_1, i_2\}, \{i_2, i_3\}, \{i_3, i_4\}\}$ being isomorphic to $P_4$ and $i_1 < \ldots < i_4$. Then by the TOMKS property of $K$, $M$ should also contain $\{\{i_1, i_3\}, \{i_1, i_4\}, \{i_2, i_4\}\}$. This is a contradiction.

Now assume that $\{\{i_1, i_3, i_5\}, \{i_1, i_4, i_6\}, \{i_2, i_3, i_6\}, \{i_2, i_4, i_5\}\}$ is the subset of one of the minors $M$ of $\mathcal C$, such that it is isomorphic to $Q_6$.
W.l.o.g. we assume that $i_1 < \ldots < i_6$. 
Then $M$ must also contain $\{i_2, i_3, i_5\}$, and this gives us a contradiction. 
For any other possible ordering of the indices, we can obtain a symmetric argument. 
\end{cpf}

From Claim~\ref{claim this one} and Theorem~\ref{theo: seymour}, we obtain the following statement: If $\mathcal C$ is the minimal cover set of a TOMKS $K$, then 
$
\conv(\{y \in \Z^n_+ \mid y(C) \geq 1, \forall C \in \mathcal C\}) = \{y \in \R^n_+ \mid y(C) \geq 1, \forall C \in \mathcal C\}
$
if and only if $\mathcal C$ has no minor isomorphic to a clutter $J_q$, with $q \geq 3$. 
Therefore, all extreme points of the polytope $\{y \in [0,1]^n \mid y(C) \geq 1, \forall C \in \mathcal C\}$ are integral if and only if $\mathcal C$ has no minor isomorphic to $J_q$. 
By substituting the variable $y$ with $1-x$, we know that $\conv(\{x \in \{0,1\}^n \mid x(C) \leq |C| - 1, \forall C \in \mathcal C\}) = \{x \in [0,1]^n \mid x(C) \leq |C| - 1, \forall C \in \mathcal C\}$ if and only if $\mathcal C$ has no minor isomorphic to $J_q$. 
Moreover, it was shown in \cite{balas1972canonical} that, $\conv(K) = \conv(\{x \in \{0,1\}^n \mid x(C) \leq |C| - 1, \forall C \in \mathcal C\})$,
and this completes the proof.
\qed \end{proof}

\section{Proof of Theorem~\ref{prop: suff_full_3}}
\label{subsec: proof_suff_full_3}

\begin{proof}
Let $\pi^T x \leq \pi_0$ be a non-trivial facet-defining inequality for $\conv(K)$. 
Then clearly we have that $\pi \in \R^n_+, \pi_0 > 0$ (see also \cite{hammer1975facet}). W.l.o.g. we can assume $\pi_0 = 1$. Denote $X: = \{x \in K \mid \pi^T x = 1\}$, which has dimension $n-1$ since $\pi^T x \leq 1$ is facet-defining. 
First, we obtain the following claims.
\begin{claim}
\label{claim: charact_1}
$\pi_i = 0$ for all $i = q+1, \ldots, n$.
\end{claim}
\begin{cpf}
For any $i \in [n] \setminus [q]$, there exists some $x' \in X$ with $x'_i = 0$, since otherwise $X \subseteq \{x \mid x_i = 1, \pi^T x = 1\},$ which has dimension $n-2$. 
If $\supp(x') \cup \{i\}$ is a cover for $K$, then it must contain some minimal cover in $J_q$. 
However, $\supp(x')$ does not contain any minimal cover (since $x'$ is feasible), and $i \in [n] \setminus [q]$ is not contained in any minimal cover. Hence we know that $x' + e^i \in K$, which has $\pi^T (x' + e^i) \leq 1$. 
Because $\pi^T x' = 1$, we obtain that $\pi_i = 0$. 
\end{cpf}

\begin{claim}
\label{claim: charact_2}
$\pi_1 = \ldots = \pi_{p-1} = \pi([n]) - 1> 0$.
\end{claim}
\begin{cpf}
Suppose $\pi_i = 0$ for some $i \in [p-1]$. From the minimal cover structure of $K$, we know that $[n] \setminus \{i\}$ is not a cover for $K$. 
Hence $\pi([n]) = \pi([n] \setminus \{i\}) \leq 1 = \pi_0$, which means that $\pi^T x \leq 1$ is dominated by the bound constraints $x_j \leq 1, \forall j \in [n]$, a contradiction.
Furthermore, for any $i \in [p-1]$, there exists some point $x' \in X$ with $x'_i = 0$, since otherwise $X \subseteq \{x \mid x_i = 1, \pi^T x = \pi_0\}$ which has dimension $n-2$. Hence $1 = \pi^T x' \leq \pi([n] \setminus \{i\}) \leq 1$, which gives $\pi_i = \pi([n]) - 1$, for any $i \in [p-1]$. 
\end{cpf}

\begin{claim}
\label{claim: charact_3}
If $\pi_p > 0$, then $\pi([p]) = 1$.
\end{claim}

\begin{cpf}
Since $\sum_{i=1}^p x_i = p-1$ is not valid for $K$, and $X$ is $(n-1)$-dimensional, we know that there exists $x' \in X$ with $x'([p]) = p$ or $p-2$. 
If $x'([p]) = p-2$, then there must exist $i,j \in [p]$ such that $x'_i = x'_j = 0$. From Claim~\ref{claim: charact_2} and the assumption of $\pi_p > 0$, we know it's impossible. 
Hence there exists $x' \in X$ with $x'_i = 1$ for any $i \in [p]$. 
From the minimal cover set $K$, we know $x'_j = 0$ for any $j \in [q] \setminus [p]$.
Therefore, $1 = \pi^T x' = \pi([p]) = 1$.
\end{cpf}

Next, we consider three different cases.
\begin{enumerate}
\item Case $\pi_p = 0$.
In this case, we want to show that $\pi^T x \leq 1$ is the same as the CI $\sum_{i=1}^{p-1} x_i + \sum_{i=p+1}^q x_i \leq q-2$.
First, we prove that $x([q]) - x_p = q-2$ for any $x \in X$. If not, then there exists some $x' \in X$ and $i,j \in [q] \setminus \{p\}$, such that $x'_i = x'_j = 0$. 
We construct a new point $x''$ from $x'$ by switching the $j$-th component from 0 to 1, and setting $x''_p = 0$.
Since $x''_p = x''_i = 0$, we know that $x'' \in K$.
Note that $1 \geq \pi^T x'' = \pi^T x' + \pi_j = 1+\pi_j$, thus we have $\pi_j = 0$. 
Hence $\pi([n] \setminus \{p, j\}) = \pi([n])$. 
However, $[n] \setminus \{p, j\}$ is not a cover, and we obtain $\pi([n]) \leq 1$, which means that the inequality $\pi^T x \leq 1$ is dominated by the bound constraints, a contradiction.
Therefore $x([q]) - x_p = q-2$ for any $x \in X$. 
Since $X = \{x \in K \mid \pi^T x = 1\}$ is assumed to have dimension $(n-1)$, we know that $\pi^T x \leq 1$ must be the same inequality as the CI: $\sum_{i=1}^{p-1} x_i + \sum_{i=p+1}^q x_i \leq q-2$.

\item Case $\pi_p > 0$, and there exists $x' \in X$ with $x'_p = 0$, such that $x'_i = 1$ for any $i \in \supp(\pi) \cap \{p+1, \ldots, q\}$.
In this case consider such point $x'$.
From Claim~\ref{claim: charact_2}, we have $\pi([n] \setminus \{i\}) = 1$ for any $i \in [p-1]$, thus we know that $x'_i = 1$ for any $i \in [p-1]$. 
Hence $\pi^T x' = \pi([n]) - \pi_p = 1$, from Claim~\ref{claim: charact_2} we have $\pi_1 = \ldots = \pi_{p-1} = \pi_p$. 
Also by Claim~\ref{claim: charact_3}, we have $\pi_1 = \ldots = \pi_p = \frac{1}{p}$. 
So the original inequality $\pi^T x \leq 1$ is just $(\frac{1}{p}, \ldots, \frac{1}{p}, \pi_{p+1}, \ldots, \pi_q, 0, \ldots, 0)^T x \leq 1$, and from Claim~\ref{claim: charact_2}, there is $\sum_{i=p+1}^q \pi_i = 1-\frac{1}{p} \cdot (p-1) = \frac{1}{p}$. 
Multiplying each CI $\sum_{i=1}^p x_i + x_j \leq p$ by a  non-negative number $\pi_j$ for each $j = p+1, \ldots, q$, and summing them up, we obtain that our facet-defining inequality $\pi^T x \leq 1$ is dominated by the CIs $\sum_{i=1}^p x_i + x_j \leq p \text{ for } j=p+1, \ldots, q$.
This means that $\pi^T x \leq 1$ coincides with one of these CIs.

\item 
Case $\pi_p > 0$, and for any $x' \in X$ with $x'_p = 0$, there is $x'_j = 0$ for some index $j \in \supp(\pi) \cap \{p+1, \ldots, q\}$.
In this case, we have the following claim.
\begin{claim}
\label{claim: charact_4}
$\pi_{p+1} = \ldots = \pi_q = \pi_1 - \pi_p >0$.
\end{claim}
\begin{cpf}
First, we show that $\pi_i > 0$ for any $i \in [q] \setminus [p]$. If not, there is $i \in [q] \setminus [p]$ such that $\pi_i = 0$. 
Arbitrarily pick a point $x' \in X$ with $x'_p = 0$. By the assumption of this case, then there must exist $j \in [q] \setminus [p]$ such that $x'_j = 0$ and $\pi_j > 0$.
Then we construct another point $x''$ from $x'$ by setting the $i$-th component to 0. 
Since $\pi_i = 0$, we have $\pi^T x'' = \pi^T x' = 1$, where $x''_p = x'_p = 0, x''_j = x'_j = 0, x''_i = 0$. 
Note that $x'' + e^j$ is also a feasible point in $K$.
However, $1 =\pi^T x'' < \pi^T (x'' + e^j)$, a contradiction. 

Next, we show that $\pi_i = \pi_1 - \pi_p$ for any $i  \in  [q] \setminus [p]$. 
Note that since $[q] \setminus \{p\}$ is a minimal cover, we know that $[q] \setminus \{p, i\}$ is not a cover for any $i \in [q] \setminus [p]$. 
Hence $\pi([q] \setminus \{p, i\}) \leq 1$. Then from Claim~\ref{claim: charact_1} and \ref{claim: charact_2}, we obtain $\pi_p + \pi_i \geq \pi_1$, for any $i \in [q] \setminus [p]$.
If for some $i' \in [q] \setminus [p]$ there is $\pi_p + \pi_{i'} > \pi_1$, then $\pi([q] \setminus \{p, i'\}) < 1$, which yields that $\sum_{i=1}^p x_i + x_{i'} = p$ for any $x \in X$.
Thus $\pi^T x \leq 1$ coincides with the CI $\sum_{i=1}^p x_i + x_{i'} \leq p$. 
However, we have shown that $\pi_i > 0$ for any $i \in [q] \setminus [p]$, and this gives a contradiction because $p \leq q-2$. 
Therefore, for any $i \in [q] \setminus [p]$, we have $\pi_p + \pi_i  = \pi_1$.
\end{cpf}
Let $\pi_1: = \lambda$, then Claim~\ref{claim: charact_1} gives $\pi_{q+1} = \ldots = \pi_n = 0$, Claim~\ref{claim: charact_2} gives $\pi_1 = \ldots = \pi_{p-1} = \lambda$, Claim~\ref{claim: charact_3} gives $\pi_p = 1-(p-1)\lambda$, and Claim~\ref{claim: charact_4} gives $\pi_{p+1} = \ldots = \pi_q = p \lambda - 1$. 
Also from Claim~\ref{claim: charact_2}, we have $\pi([n]) - \lambda = 1$, hence: $\lambda = (q-p)(p \lambda - 1)$, which gives $\lambda = \frac{q-p}{p(q-p)-1}$. Therefore, the original facet-defining inequality $\pi^T x \leq 1$ coincides with $ (q-p) \sum_{i=1}^{p-1} x_i + (q-p-1)x_p + \sum_{i=p+1}^q x_i \leq p(q-p)-1$. 
Note that for the multi-cover $\{\{1, \ldots, p-1, p+1, \ldots, q\}, \{1,\ldots, p,q\}\}$, the inequality has the same structure as the one given in Example~\ref{exam: 3}, and the corresponding MCI is $ (q-p) \sum_{i=1}^{p-1} x_i + (q-p-1)x_p + \sum_{i=p+1}^q x_i \leq p(q-p)-1$.
\end{enumerate}
The above discussion concludes the proof. 
\qed \end{proof}

\section{Separation Formulation in Section~\ref{sec: separation}}
\label{subsec: additional_formulation_sep}

Let $\C$ be a cover-family.
From the definition of multi-cover, $\C$ is a multi-cover if and only its skeleton $\mathscr{S}$ satisfies the property: For any $ T \subseteq \cup_{S \in \mathscr S} S$, there exists some $S' \in \mathscr S$ such that $T$ is comparable with $S'$.
Henceforth all skeletons are assumed to have such property.
Now for a given skeleton $\mathscr{S} := \{S_1, \ldots, S_k\}$ and a fractional solution $\tilde x$, we consider the following MIP. 
We let $S := \cup_{i=1}^k S_i$ and $\bar S_h := S \setminus S_h$ for any $h \in [k]$. 

\begin{align}
\label{Sep}
\tag{Sep-$\S$}
\begin{split}
\min \quad & t + \sum_{i \in [n]} \gamma_i - \sum_{i \in [n]} (\gamma_i + \sum_{s \in S} \alpha^s_i) \tilde x_i \\
\st \quad & t \geq \sum_{i \in [n]} \sum_{s \in S_h} \alpha^s_i, \quad \forall h \in [k]\\
& \sum_{s \in S} u^s_i + w_i \leq 1, \ \alpha^s_i \leq M u^s_i, \ \gamma_i \leq M w_i \quad \forall i \in [n] \\
& \sum_{i \in [n]} u^s_i = 1, \quad \forall s \in S \\
& \alpha^s_i \geq \alpha^{s'}_j + (1+M) u^s_i -M, \quad \forall i,j \in [n], j > i, s \in S, s' \in \Pi(s) \\
& w_i = \sum_{h \in [k]} w^h_i, \quad i \in [n] \\
& \gamma_i \geq \sum_{j>i} \sum_{s \in \bar S_h} \alpha^s_j + (1+n |\bar S_h| M)w^h_i - n|\bar S_h| M, \quad \forall i \in [n], h \in [k] \\
& \gamma_i \geq \sum_{s \in \bar S_h} \alpha^s_j + |\bar S_h| M w^h_i -  |\bar S_h| M, \quad \forall i,j \in [n], j < i, h \in [k] \\
& \sum_{i \in [n]} (\sum_{s \in S_h} u^s_i + w_i) A_{j,i} \geq (b_j+1) \lambda^h_j, \quad \forall h \in [k], j \in [m] \\
& u^s_i + \sum_{j<i} u^{s+1}_j \leq 1, \quad \forall s = 1, \ldots, |S| - 1 \\
& \sum_{j=1}^m \lambda^h_j \geq 1, \quad \forall h \in [k] \\
& u^s_i \in \{0,1\}, \alpha^s_i \in \Z,\ w_i, w^h_i, \lambda^h_j \in \{0,1\}, \quad \forall s \in S, i \in [n], j \in [m], h \in [k].
\end{split}
\end{align}

Here $\Pi(s): = \{s' \in S \mid \exists h \in [k], \st \ s \in S_h, s' \notin S_h, s' > s\}.$ 
Then, the separation problem of MCIs with skeleton $\mathscr{S}$ can be solved exactly using the MIP \eqref{Sep}.

\begin{theorem}
\label{theo: general_skeleton_sep}
Given a TOMKS $K$. For any given skeleton $\S$ and a point $\tilde x$,
there exists a multi-cover $\C$ whose skeleton is $\S$ and an associated MCI that separates $\tilde x$ from $K$ if and only if \eqref{Sep} admits an optimal value less than 1.
\end{theorem}

\begin{proof}
In \eqref{Sep}, binary $u^s_i$ denotes whether or not the index $i$ of variables corresponds to the index $s$ in the skeleton; $\alpha^s_i$ denotes the MCI coefficient of variable $x_i$, when $u^s_i$ = 1; Binary $w_i$ denotes whether or not variable index $i$ appears in the intersection of the multi-cover; $\gamma_i$ denotes the MCI coefficient of variable $x_i$, when $w_i$ = 1; For any $h \in [k]$ and $j \in [m]$, binary $\lambda^h_j$ denotes whether or not the cover $C_h$ corresponding to the skeleton $S_h$ violates the $j$-th knapsack constraint of the problem; $ t + \sum_{i \in [n]} \gamma_i $ represents the maximum value of $\alpha(C_h), h \in [k]$. 
Therefore, the associated MCI is represented by the inequality:
$$
 \sum_{i \in [n]} (\gamma_i + \sum_{s \in S} \alpha^s_i) x_i \leq t + \sum_{i \in [n]} \gamma_i - 1.
$$
Hence the optimal value of \eqref{Sep} is strictly less than 1 if and only if the MCI is a separating inequality. 

Now we verify that the meaning of those variables can be formulated by the constraints in \eqref{Sep}. 
For each index $s \in S$, since it only corresponds to one variable index $i \in [n]$, so we have constraint $\sum_{i \in [n]} u^s_i = 1, \forall s \in S$. Similarly, for each index $i \in [n]$, since it either appears in the intersection of the covers, or corresponds to a single index $s \in S$, or is not contained by any cover, so we have constraint $ \sum_{s \in S} u^s_i + w_i \leq 1, \forall i \in [n]$. Constraint $\alpha^s_i \geq \alpha^{s'}_j + (1+M) u^s_i -M, \forall j > i, s \in S, s' \in \Pi(s)$ formulates the Step~\ref{step: MCI_4} of Algorithm~\ref{alg:mci}, and constraints $\gamma_i \geq \sum_{j>i} \sum_{s \in \bar S_h} \alpha^s_j + (1+n |\bar S_h| M)w^h_i - n|\bar S_h| M, \forall i \in [n], h \in [k],  \gamma_i \geq \sum_{s \in \bar S_h} \alpha^s_j + |\bar S_h| M w^h_i -  |\bar S_h| M, \forall j < i, h \in [k]$ as well as $w_i = \sum_{h \in [k]} w^h_i$ formulate the Step~\ref{step: modified_coe_intersection} of Algorithm~\ref{alg:mci}. Constraint $u^s_i + \sum_{j<i} u^{s+1}_j \leq 1, \forall s = 1, \ldots, |S| - 1$ formulates the bijective relation between skeleton and the discrepancy family of multi-cover: if index $i \in [n]$ corresponds to the skeleton index $s \in S$, then for any skeleton index $s' > s$, it only corresponds to the index $j > i$. The remaining constraints are easy to interpret. 
\qed \end{proof}

\section{Numerical Results Report}
\label{subsec: detailed_data}
The following two tables present the detailed results for the optimality gap we obtained from solving different linear relaxation problems. 

\begin{table}[H]
\centering
\vspace{-0.4cm}
\begin{tabular}{  p{1.6cm} | p{0.8cm} | p{1cm} p{1cm} p{1cm} | p{1cm} p{1cm} p{1cm} }
\hline
(n, m, seed) & LP   & MCI & E-MCI & L-MCI &  CI & ECI & LCI \\ 
\midrule
(20, 1, 1) & 1.05   & 0.81   &  0.58 & 0.48  & 0.87  &  0.59 &  0.59 \\
(20, 1, 2) & 3.7   & 1.92  & 0.17  &  0 & 2.55  &  0.82 &  0.82 \\
(20, 1, 3) &  0.87  &  0.24 & 0.24  & 0.22  & 0.68  & 0.63  & 0.63  \\
(20, 1, 4) & 1.78   &  0 & 0  & 0  & 1.00  & 0.16  & 0  \\
(20, 1, 5) & 6.15   &  1.06 & 1.06  & 0.14  &  2.79 & 2.07  & 2.07  \\
(20, 1, 6) & 9.90    & 5.69  & 0.22  & 0  & 6.47  & 1.01  & 1.01  \\
(20, 1, 7) &   7.29  & 0  &  0 & 0  &  0 &  0 & 0  \\
(20, 1, 8) & 3.06   & 0.41  & 0.40  & 0.29  & 1.99  & 0.54  &  0.54 \\
(20, 1, 9) & 1.48   & 0.65  &  0 & 0 & 0.87  & 0.49  & 0.49  \\
(20, 1, 10) & 4.03   & 0  & 0  & 0  & 2.81  &  0.15 & 0.15  \\
\midrule
average & 3.93   & 1.08  & 0.27  & 0.11  & 2.00  & 0.65  & 0.63  \\
\bottomrule
(20, 2, 1) & 3.37   & 1.79  & 0.89  & 0.60  & 1.83  & 1.50  & 0.83  \\
(20, 2, 2) &  12.50   & 7.18  &  3.16 & 0 & 8.68  & 4.98  &  0 \\
(20, 2, 3) &  1.45  &  0.27 & 0.05  & 0  & 0.81  & 0.10  & 0.08  \\
(20, 2, 4) & 3.96   &  1.66 &  0.65 & 0.10  &  2.79 & 0.83  &  0.83 \\
(20, 2, 5) & 3.15   &  1.65  & 1.39   & 1.03  & 2.04  & 1.65  &  1.65 \\
(20, 2, 6) &  5.43  & 1.73  & 0  & 0  & 4.34  &  0 &  0 \\
(20, 2, 7) &  9.63  & 3.69  & 0.23  & 0  &6.36  & 0.67  & 0.67  \\
(20, 2, 8) &  2.49  & 0.50  &  0.46 &  0 &  1.52 &  1.06 &  1.06 \\
(20, 2, 9) &   14.64  & 8.82  &  2.66 & 0  & 10.67  & 2.83  &  2.83 \\
(20, 2, 10) &  3.45  & 1.67 & 0.73  & 0.41  & 3.24   & 3.24  & 3.24  \\
\midrule
average & 6.01   &2.90   & 1.02  &  0.21 &4.23  & 1.69  &  1.12 \\
\bottomrule
(20, 3, 1) & 5.53  & 3.27  & 0  &  0 & 3.27  & 0  &  0 \\
(20, 3, 2) &  4.46  & 2.88  &  1.64 &  0 &  3.88 & 3.08  & 3.08  \\
(20, 3, 3) &  25.31  &  19.53 & 0  &  0 &  19.80 & 0   & 0  \\
(20, 3, 4) & 3.77   & 3.03  & 1.74  &  1.64 & 3.60 & 2.00  &  2.00 \\
(20, 3, 5) &  4.97  &  2.32 & 1.46  &  0 &  4.12 & 1.67  & 1.67  \\
(20, 3, 6) &  1.87  &0   & 0  & 0  & 0.28  &  0.12 &  0.12 \\
(20, 3, 7) & 3.72   & 0.88  & 0  & 0  &  3.15 &  1.46 & 1.46  \\
(20, 3, 8) &  3.31  &  1.66 & 0.41  &  0 & 1.96  & 0.83  & 0.83  \\
(20, 3, 9) &  1.29  & 0.87  & 0.44  &  0 & 1.08  &  0.62 & 0.16  \\
(20, 3, 10) & 6.11   & 4.42  & 1.58  &  1.09 &  5.53 & 1.90  &  1.90 \\
\midrule
average &  6.03  & 3.89  & 0.73  &  0.27 &  4.67 & 1.17  & 1.12  \\
\bottomrule
 \end{tabular}
\end{table}

\begin{table}[H]
\centering
\begin{tabular}{  p{1.6cm} | p{0.8cm} | p{1cm} p{1cm} p{1cm} | p{1cm} p{1cm} p{1cm} }
\hline
(n, m, seed) & LP   & MCI & E-MCI & L-MCI &  CI & ECI & LCI \\ 
\midrule
(30, 1, 1) & 7.95   &  5.29 & 0.64  &0   &  6.05 &  0.76 &  0.76 \\
(30, 1, 2) &  2.48  &  0.91 & 0.20  &  0 & 1.38  &  0.25 & 0.25  \\
(30, 1, 3) &  24.75  & 23.35  & 0.70  & 0  &  24.24 & 0.70  &  0.70 \\
(30, 1, 4) &  2.84  & 1.37  & 1.27  & 0.81  &  1.48 & 1.48  &  1.48 \\
(30, 1, 5) & 1.66   & 0  & 0   & 0  & 0  &  0 & 0 \\
(30, 1, 6) &   9.45 &  6.74 & 1.53  & 0.77  & 7.60  &  1.84 &   1.84 \\
(30, 1, 7) &1.43    & 1.27  & 1.17  & 0.83  &  1.31 &  1.25 &  1.25 \\
(30, 1, 8) &3.81    &  2.84 & 0.42  & 0  & 3.17  & 0.47  &  0.22 \\
(30, 1, 9) &  6.20  & 4.66  & 0  & 0  &  5.23 &  0 &  0 \\
(30, 1, 10) &  2.83  &  1.34 & 0  &  0 & 1.91  &  0 &  0 \\
\midrule
average & 6.34  & 4.78 & 0.59  &0.24  &  5.24 & 0.68  & 0.65 \\
\bottomrule
(30, 2, 1) & 5.73  & 4.42  &  0.67 &0  &  4.78 & 0.69  & 0.69 \\
(30, 2, 2) & 5.14  &  3.25 & 3.25  & 0.30  &  4.05 & 3.78  & 3.78 \\
(30, 2, 3) & 9.94  &7.20   & 0  & 0 & 8.11  &  0 & 0  \\
(30, 2, 4) &  3.02 & 0.63  & 0  & 0 &  1.38 &0  &  0 \\
(30, 2, 5) & 2.98  & 2.46  &0   & 0  & 2.46  &  0 & 0 \\
(30, 2, 6) & 4.93  & 2.78  & 0  & 0 & 3.42  & 0  & 0 \\
(30, 2, 7) &  3.40 & 2.10  & 1.99  &0  &  2.80 & 2.80  & 0.74  \\
(30, 2, 8) & 5.86  & 5.06  & 1.35  & 0.20 &  5.45 & 1.35 & 1.35 \\
(30, 2, 9) &  3.04 &  0 &  0 &0  &  0.86 &  0.58 & 0  \\
(30, 2, 10) & 3.59  &  1.61 & 1.61  & 1.37 & 1.77  & 1.73  & 1.73 \\
\midrule
average & 4.76  & 2.95 & 0.89  &0.19  &  3.51 & 1.09  & 0.83 \\
\bottomrule
(30, 3, 1) & 3.76  & 2.94  &  1.19 & 0 & 3.14  & 1.32  & 0.71  \\
(30, 3, 2) & 4.13  &  3.26 & 0.37  & 0 &  3.68 & 0.38  & 0.19 \\
(30, 3, 3) &14.64   &  13.62 & 3.92  & 1.67 & 13.98   & 4.04  & 4.04  \\
(30, 3, 4) & 1.57  & 1.49  &  1.37 &0.89  & 1.55  & 1.41  & 1.41 \\
(30, 3, 5) & 1.09  &1.00 &  0.60 & 0.37 &  1.05 & 0.61  & 0.42 \\
(30, 3, 6) & 1.46  &   0& 0  & 0 & 0 &  0 & 0 \\
(30, 3, 7) &  0.99 &  0.88 & 0.11  & 0 & 0.92  & 0.11  & 0.11 \\
(30, 3, 8) & 1.23  &  0.55 & 0  & 0  &  0.80 & 0  & 0  \\
(30, 3, 9) &0.99  &  0.77 & 0.41 & 0.19 &   0.83 & 0.51  & 0.51 \\
(30, 3, 10) &  1.46 & 0.53  & 0.37  & 0 &  0.70 &  0.60 & 0.60 \\
\midrule
average & 3.13  & 2.50 & 0.83  &0.31  &  2.67 & 0.90  & 0.80 \\
\bottomrule
 \end{tabular}
\end{table}

\end{appendices}

%
%

\end{document}